\documentclass{article}
\usepackage[T1]{fontenc}
\usepackage[utf8]{inputenc}
\usepackage[american]{babel}

\usepackage{graphicx}

\usepackage{amsmath,amsfonts,enumerate,amssymb}
\usepackage{amsthm}

\usepackage[usenames,dvipsnames,svgnames]{xcolor}

\usepackage{enumitem}

\usepackage{csquotes}

\newcommand*{\Tn}{\mathcal{T}_n}
\newcommand*{\CC}{\mathbb{C}}

\newcommand*{\cT}{\mathcal{T}}

\newcommand*{\LS}{L}

\newcommand*{\oLS}{\overline{L}}

\newcommand*{\oSi}{\overline{S}_{[0,1]}}
\newcommand*{\Si}{S_{[0,1]}}

\newcommand*{\oLStilden}{\overline{L}^{(\widetilde{n})}}
\newcommand*{\oLSn}{\overline{L}^{(n)}}

\newcommand*{\cless}{\prec}
\newcommand*{\clesse}{\preccurlyeq}

\newcommand*{\diff}{\Phi}

\newcommand*{\newJ}{\widetilde{J}}
\newcommand*{\newF}{\widetilde{F}}
\newcommand*{\newmstar}{\widetilde{m}^*}

\newcounter{thmcounter}
\theoremstyle{plain}
\newtheorem{thm}[thmcounter]{Theorem}
\newtheorem{theorem}[thmcounter]{Theorem}

\theoremstyle{plain}

\newtheorem{proposition}[thmcounter]{Proposition}

\newtheorem{lemma}[thmcounter]{Lemma}
\newtheorem{cor}[thmcounter]{Corollary}

\theoremstyle{definition}
\newtheorem{rem}[thmcounter]{Remark}

\theoremstyle{example}
\newtheorem{example}[thmcounter]{Example}

\newcommand*{\ZZ}{\mathbb{Z}}

\newcommand{\avec}{\mathbf{a}}

\newcommand*{\uR}{\underline{\mathbb{R}}}

\newcommand*{\oS}{\overline{S}}
\newcommand*{\ve}{\varepsilon}
\newcommand*{\de}{\delta}

\newcommand*{\ff}{\varphi}

\newcommand*{\RR}{{\mathbb{R}}}

\newcommand*{\cI}{\mathcal{I}}
\newcommand*{\cJ}{\mathcal{J}}

\newcommand*{\si}{\sigma}

\newcommand*{\bR}{\mathbb{R}}

\newcommand*{\zz}{\mathbf{z}}

\newcommand*{\ww}{\mathbf{w}}
\newcommand*{\mol}{\overline{m}}
\newcommand*{\mul}{\underline{m}}

\newcommand*{\yy}{\mathbf{y}}
\newcommand*{\ee}{\mathbf{e}}

\newcommand*{\xx}{\mathbf{x}}

\newcommand*{\TT}{\mathbb{T}}
\newcommand*{\NN}{\mathbb{N}}
\def\downto{\downarrow}
\def\upto{\uparrow}

\newcommand*{\bT}{\mathbb{T}}

\newcommand*{\bC}{\mathbf{C}}

\title{On the weighted trigonometric Bojanov-Chebyshev extremal problem}

\author{
Béla Nagy and Szilárd Gy. Révész}

\begin{document}

\maketitle

Dedicated to Vitalii Vladimirovich Arestov, a leading mathematician and teacher of generations, on the occasion of his 80th anniversary

\begin{abstract}
We investigate the weighted Bojanov-Chebyshev extremal problem for trigonometric polynomials, that is,
the minimax problem of minimizing $\|T\|_{w,C(\TT)}$,
where $w$ is a sufficiently nonvanishing, upper bounded, nonnegative weight function,
the norm is the corresponding weighted maximum norm on the torus $\TT$, and $T$ is a trigonometric polynomial
with prescribed multiplicities $\nu_1,\ldots,\nu_n$ of root factors $|\sin(\pi(t-z_j))|^{\nu_j}$.
If the $\nu_j$ are natural numbers and their sum is even,
then $T$ is indeed a trigonometric polynomial
and
the case when all the $\nu_j$ are 1 covers the Chebyshev extremal problem.

Our result will be
more general, allowing, in particular,
so-called generalized trigonometric polynomials.
To reach our goal, we invoke Fenton's
sum of translates method.
However, altering from the earlier described cases without weight or on the interval,
here we find different situations, and can state less about the solutions.
\end{abstract}



\section{Introduction}\label{sec:Intro}

In this paper our aim is to solve the Bojanov-Chebyshev
extremal problem
in the setting of weighted trigonometric polynomials.
As in our earlier papers on related subjects,
our approach is the so-called
\enquote{sum of translates method} of Fenton,
what he introduced in \cite{Fenton}.
However, here we do not develop the whole theory for two
reasons:
first, in the periodic,
i.e., torus setup,  in
\cite{TLMS2018}
we have already developed much of what is
possible, and second, much of what we could find useful
here, is simply not holding true.
In this regard, our
Example \ref{ex:counterexample} is an important part of the
study, showing the limits of any proof in this generality.

The analogous problems for the unweighted periodic case and
the weighted and unweighted algebraic polynomial cases on
the interval were already solved in \cite{TLMS2018} and in
\cite{Sbornik} and \cite{JMAA}.
The results available to
date do not imply, not in a direct and easy way,
the corresponding result to the weighted trigonometric
polynomial Bojanov-Chebyshev problem.
In fact, some of them
simply does not remain valid.
So, the weighted
trigonometric polynomial case poses new challenges and
requires a careful adaptation of our methods, with an
avoidance of certain obstacles -- for example a
bagatelle-looking, but in fact serious dimensionality
obstacle in the way of proving a homeomorphism theorem,
analogous to the earlier cases -- and recombining our
existing knowledge about the torus setup with all what can
be saved and reused from the interval case.
So, we heavily
rely on all our earlier papers \cite{TLMS2018} \cite{Homeo}
\cite{Sbornik} \cite{Ural} \cite{JMAA} on the subject, while
these in themselves will not suffice to reach our goals.
We will still need to devise new proofs or at least new
versions for various existing arguments.

We note that proving minimax- and equioscillation type results in certain
contexts may be attempted without rebuilding the whole
theory, just by transferring some existing results of an
already better explored case to the new settings.
This has
already been done in \cite{TLMS2018} for the (unweighted)
algebraic polynomial case of the interval, deriving it from
the (unweighted) trigonometric polynomial case, explored in
the major part of \cite{TLMS2018}.
However, the transference was not easy and broke down for
general weights (even if for even weights it seemed
working).
Similarly, in \cite{Tatiana} Tatiana Nikiforova
succeded in transferring certain results to the real line
and semiaxis cases from the interval case -- while leaving
unresolved some of the related and
still interesting questions.
However, in both cases we can expect a more
detailed and complete picture when we take the time to build
up the method and explore the full strength
of it right in the given context.
Therefore, we did not settle with the
results which could be transferred from \cite{TLMS2018}, but
worked out the interval case fully
in \cite{Sbornik}, \cite{Homeo}, \cite{Ural};
we also think that it would be
worthwhile to do so in the cases of the real line and the
semiaxis.
In particular, the relevant variant of the
homeomorphism theorem is missed very much for the real line
and the semiaxis.
However, as already said and as will be
explained in due course later, in the current setup that
buildup does not seem to be possible, and we must be
satisfied by a combination of transferred results and ad hoc
arguments.

From our point of view, however, the weighted trigonometric
polynomial Bojanov-Chebyshev problem is not a main goal,
but more of an application,
which testifies the strength of the method.
We try to work in a rather general framework,
and prove more general results than that.
In particular, the results will be valid
also for generalized trigonometric
polynomials (GTPs), which are introduced,
e.g., in
\cite{ErdelyiBorwein} Chapter A4 as follows.
\begin{align}\label{def:gtp}
\cT:=\bigcup_{n=1}^\infty \Tn, \qquad \Tn:= \bigg\{ T(\zz,t) & :=
c_0\prod_{j=1}^{n} \left| \sin\pi(t-z_j)\right|^{\nu_j} ~:~ c_0> 0,
\\ \notag & \nu_j>0 \, (j=1,\ldots,n), \, \zz=(z_1,\ldots,z_n)\in\CC^n \bigg\}.
\end{align}
By periodicity one can assume $0\le \Re z_j<1$, and in the below extremal problems
it is obvious that replacing $\Re z_j$ for $z_j$ can only decrease the quantity to be minimized,
so that we will assume that all the $z_j$'s are real.
However, fixing the ordering of $z_j$ (or $\Re z_j$) has a role,
with different fixed orderings posing separate extremal problems, and the
ordering-specific solution being much stronger, than just a \enquote{global} minimization.
As this issue has already been discussed in \cite{TLMS2018} and \cite{Sbornik}, e.g.,
we leave the details to the reader simply addressing the order-specific, stronger question here.
One particular result ahead of us will be the following.

\begin{theorem}
\label{thm:Bojanov}
Let $n\in \NN$ and $\nu_1,\ldots,\nu_n>0$ be given.
Put ${\bf \nu}:=(\nu_1,\ldots,\nu_n)$.

Further, let $w:\bR\rightarrow[0,\infty)$ be an upper bounded,
nonnegative, $1$-periodic weight function,
attaining positive values at more than $n$ points of $\TT:=\RR/\ZZ$.

Denote the weighted sup norm by $\|.\|_w$ and consider the minimax problem
\begin{align*}
M:=M(w,{\bf \nu}):=
\inf\big\{ & \|T(\zz,.) \|_w:\ T(\zz;t) =
\prod_{j=1}^{n} \left|\sin\pi(t-z_j)\right|^{\nu_j} \in \Tn,
\\ & \exists c \in [0,1) \quad \text{ such that }
\quad c\le \Re z_1\le \dots \le \Re z_n \le c+1 \big\}.
\end{align*}

Then there exists a minimax point $\zz^*=(z_1^*,\ldots,z_n^*)\in \TT^n$
and with the prescribed cyclic ordering of the nodes,
satisfying $\|T(\zz^*,.)\|_w=M(w,{\bf \nu})$.

Moreover, all $z^*_j$s are distinct and real, and their cyclic ordering is strict in the sense that
there exists $c \in \RR$ such that
$c <z_1*<\dots<z_n^*<c+1$
(as in the prescribed order, but with strict inequalities).

Furthermore, this extremal point has the \emph{equioscillation property}, that is,
$\max\{T(\zz^*,t):\ z^*_1\le t\le z^*_{2}\}=\dots=\max\{T(\zz^*,t):\ z^*_j\le t\le z^*_{j+1}\}= \dots= \max\{T(\zz^*,t):\ z^*_n\le t\le z^*_{1}+1\}=M(w,{\bf \nu})$.
\end{theorem}

The occurrence of $c$ in the description is another simple-looking,
yet important difference between the interval and torus setup.
Basically, we fix here the ordering of nodes $z_j^*$ only cyclically, that is,
as they follow each other when one covers the circle once, moving continuously from
some appropriate $c\in \TT$ in the positive (counter-clockwise) orientation until return.

\section{Basics for the Bojanov-Chebyshev problem}\label{sec:Basics}

\subsection{Trigonometric polynomials and generalized trigonometric polynomials}
It is well known that (real) trigonometric polynomials
can be factorized as follows.
Let
\begin{equation}
T(t):=
a_0 +\sum_{j=1}^n
a_j  \sin( 2\pi\,j\,t)
+ b_j \cos(2\pi\,j\,t)
\end{equation}
be a (real) trigonometric polynomial
($a_0,a_1,b_1,a_2,b_2,\ldots,a_n,b_n\in\bR$, $a_n^2+b_n^2\ne 0$)
of degree $n$
(with period $1$).
Then (see, e.g., \cite{ErdelyiBorwein}
p.~10) there exist uniquely $c_0\in\bR$,
$c_0\ne 0$,
$z_1,\ldots,z_{2n}\in\bC$ such that nonreal $z_j$'s occur in conjugate pairs and
\begin{equation}
T(t)=c_0\prod_{j=1}^{2n} \sin\left(\pi(t-z_j)\right).
\end{equation}
This explains that GTPs are indeed generalizations of trigonometric polynomials\footnote{The somewhat curious fact is that root factorization of trigonometric polynomials relies on \emph{pairs of factors} of the form $\sin(\pi(t-z_j))$, where one such factor in itself is not a trigonometric polynomial (as it is only antiperiodic, but not periodic by 1). Considering general products of root factors thus leads to GTPs.}.

By taking logarithm of a generalized trigonometric polynomial \eqref{def:gtp},
we have
\[
\log T(t) =
\log c_0 +
\sum_{j=1}^n \nu_j
\log\big|\sin(\pi(t-z_j))\big|.
\]
In this work we assume for normalization that our trigonometric polynomial or GTP is \emph{monic}, i.e., the
\enquote{leading coefficient} is $c_0=1$. However, we consider weights,
which are fixed, but can as well be constants, so that the weighted norm can incorporate any other prescribed leading coefficient as well.
Obviously, the weighted minimax problem is equivalent to minimizing $\log \|T(\zz,\cdot)\|_w=\sup_{\TT} \log w(t) + \sum_{j=1}^n \nu_j \log|\sin(\pi(t-z_j))|$.
This reformulation leads to considering sums (and positive linear combinations) of translated copies of the basic
\enquote{kernel function}
$\log |\sin(\pi t)|$, instead of products of root factors.
That reformulation, so standard in logarithmic potential theory, will be the starting point of our presentation of the Fenton method for the current setup.

\subsection{Basics of Fenton's sum of translates approach}

In this paragraph we present the by now standard notations and terminology how we use Fenton's method.
There is particular need for this clarification because we will use it in \emph{two} setups, needing it for the torus $\TT$,
but also time to time referring to and invoking into our arguments corresponding results for the interval case.
Here we start with notations and terminologies which can be equally interpreted for the torus and real line case,
so with a slight abuse of notation we do not distinguish between them.

However, in the next subsection we set a separated
terminology with quantities for the periodic case denoted by
a star, because in these notions there are some essential
alterations.
With this long, and sometimes doubled list of
definitions, notions and terminology, these paragraphs will
be boring and longish, but in later sections we will need it
for precise references.
Note that in most of our definitions we will \emph{not}
assume that the considered functions and setups were
periodic, and handle the periodic cases only as special
cases, especially pointing out the periodicity assumption.

A function $K:(-1,0)\cup (0,1)\to \bR$ is called a \emph{kernel function}
if it is concave on $(-1,0)$ and on $(0,1)$, and
if it satisfies
\begin{equation}\label{eq:Kzero}
\lim_{t\downto 0} K(t) =\lim_{t\upto 0} K(t).
\end{equation}
By the concavity assumption these limits exist, and a kernel function has one-sided limits also at $-1$ and $1$. We set
\begin{equation*}
K(0):=\lim_{t\to 0}K(t),\quad K(-1):=\lim_{t\downto -1} K(t) \quad\text{and}\quad K(1):=\lim_{t\upto 1} K(t).
\end{equation*}
We note explicitly that we thus obtain the extended continuous function $K:[-1,1]\to \bR\cup\{-\infty\}=:\uR$, and that we have $\sup K<\infty$.
Also note that a kernel function is almost everywhere differentiable.

A kernel function $K$ is called \emph{singular} if
\begin{equation}\label{cond:infty}
\tag{$\infty$}
K(0)=-\infty.
\end{equation}

We say that the kernel function $K$ is \emph{strictly concave}
if it is strictly concave on both of the intervals $(-1,0)$ and $(0,1)$.

In this paper we consider only systems of kernels which are constant multiples of each other,
i.e.,
\begin{equation}
K_j(t)=\nu_j K(t)
\label{Kconstmultiple}
\end{equation}
for some $\nu_1,\ldots,\nu_n>0$ and some kernel function $K(t)$.

The condition
\begin{equation}
K'(t)-K'(t-1)\geq c \quad \text{for a.e. } t\in[0,1],
\label{cond:pmc}
\tag{$PM_c$}
\end{equation}
was called \enquote{periodized $c$-monotonicity}
in \cite{Homeo} and \cite{Sbornik}.
The particular case $c=0$ deserves special attention.
Then we have
\begin{equation}
K'(t)-K'(t-1)\geq 0 \quad \text{for a.e. }
t\in[0,1].
\label{cond:pm0}
\tag{$PM_0$}
\end{equation}

Our main objective is  the study of kernels which extend to $\bR$ $1$-periodically:
\begin{equation}
K(t-1)=K(t),
\quad t\in\bR,
\label{Kperiodic}
\end{equation}
but sometimes we will invoke more general, not necessarily periodic kernels, too.
It is straightforward that \eqref{Kperiodic} implies \eqref{cond:pm0}.

\smallskip

Note that the log-trigonometric kernel
\begin{equation}
K(t):=\log\big|\sin(\pi t)\big| \qquad (t\in \RR),
\label{logsinkernel}
\end{equation}
which is in the focus of our analysis, is periodic \eqref{Kperiodic},
strictly concave and singular \eqref{cond:infty} (and in particular $K(1)=K(-1)=-\infty$, too).

\smallskip

We will call a function $J:\bR\to\uR$ an
\emph{external $n$-field function on $\bR$}
if it is $1$-periodic, bounded above and
it assumes finite values
at more than  $n$ different points from $[0,1)$.
Comparing this definition
with that of \cite{Sbornik},
we see that
if $J$ is an external $n$-field function on $\bR$, then
$J$ is an external $n$-field function on $[0,1]$.
In the opposite direction,
if $J$ is an external $n$-field function on $[0,1]$ and $J(0)=J(1)$,
then it can be extended
$1$-periodically to $\bR$,
to an external $n$-field function on $\bR$.
We use external $n$-field
functions on $\bR$
from now on and for
simplicity, we call them
field functions.
With a slight abuse of notation we will also consider them as external field functions on $\TT:=\RR/\ZZ$: again, on $\TT$ the defining properties are that $J:\TT\to \uR$, $J>-\infty$
  at
more than $n$ points of $\TT$, and $J$ is upper bounded.

\medskip

For a field function $J$
we define its
\emph{singularity set}
and \emph{finiteness domain} by
\begin{equation}\label{eq:Xdef}
X:=X_J:=J^{-1}(\{-\infty\})\cap [0,1)
\quad\text{and}\quad
X^c:=[0,1)\setminus X=J^{-1}(\RR)\cap [0,1).
\end{equation}
Then $X^c$ has cardinality exceeding $n$, in particular $X\neq [0,1)$.
Considering $J$ as defined on $\TT$,
we can replace $[0,1)$ with $\TT$ in all the above.

\medskip

Given $n\in \NN$ and $n$ kernel functions $K_1,\ldots,K_n$, and an $n$-field function $J$,
pure sum of translates and
sum of translates functions are defined as
\begin{align}
f(\xx,r):=\  &
\sum_{j=1}^n
K_j\left(r-x_j\right),
\\
F(\xx,r):=\ &
J(r) + f(\xx,r)
\label{sumoftranslates}
\end{align}
where $r\in\bR$,
 and
$\xx=(x_1,x_2,\ldots,x_n)\in [0,1]^n$; or,
analogously, we can define $f(\yy,t)$ and $F(\yy,t)$ also for $t\in \TT$, $\yy \in \TT^n$.

\subsection{Differences between the torus and the interval setting}

A useful step in several of our arguments--already in \cite{TLMS2018}--is the
\enquote{cutting up}
of the torus at an arbitrary point $c\in \TT$.
To formalize it, we introduce the mapping $\pi_c: \RR \to \TT$, $\pi_c(r):=\{r+c\}= r+c \mod 1$.
That constitutes a (multiple) covering mapping of $\TT$ (hence in particular it is continuous), and it is bijective on $[0,1)$,
so  its inverse $\pi_c^{-1}:=(\pi_c|_{[0,1)})^{-1} : \TT \to [0,1)$ is bijective, too.
However, in the inverse direction the mapping ceases to remain continuous:
it is continuous at all $t\ne c , t\in \TT$,
but at $c$ it has a jump.

Cutting up is particularly useful when we want to prove local results like e.g. continuity of some mappings at $\xx \in \TT^n$.
Choosing $c$ appropriately, node systems from $\TT^n$,
subject to some ordering restriction and close to $\xx$ may correspond to node systems in $[0,1)^n$ admitting a specific ordering in the interval.
However, ordered node systems do not have a global match on $\TT$.
We will detail this phenomenon below.

In \cite{TLMS2018} we introduced,
for any permutation $\sigma: \{1,\ldots,n\} \to \{1,\ldots,n\}$,
the corresponding simplex on $[0,1]$ as
\begin{equation*}
\Si^{(\sigma)}:=
\{(x_1,\ldots,x_n)\in[0,1]^n:
0< x_{\si(1)}< x_{\si(2)}< \ldots < x_{\si(n)}< 1 \}\subset \bR^n.
\end{equation*}
Its closure is
\begin{equation*}
\oSi^{(\sigma)}=\{(x_1,\ldots,x_n)
\in[0,1]^n: \  0\le x_{\si(1)} \le x_{\si(2)}\le \ldots \le x_{\si(n)}
\le
1 \}\subset \bR^n.
\end{equation*}

An essential difference between the torus and interval setup is that in $[0,1]$
we cannot perturbe in both directions the nodes lying at the endpoints:
they can be moved only towards
the interval center.
That restriction could be considered responsible for the need of
some monotonicity assumption about the kernels when proving minimax etc.~results for the interval case.
However, in the torus a monotonicity assumption is in fact impossible:
a periodic and monotone kernel function would necessarily be constant.
On the other hand we had already seen in \cite{TLMS2018} that perturbation of node systems on the torus,
in particular when we are free to decide about the direction of change of the nodes, are very useful.
As a consequence, we need to consider all node systems,
which may arise by means of such a perturbation, together.
So, we need to consider the case, when a node passes over $0$ and reappears at $1$, as the same ordering.
In fact, that is very natural: on the torus there is no strict ordering, but only an orientation, and the
\enquote{order of nodes}
can only be fixed as up to rotation.
This we may call cyclic ordering. In this sense we may write $x_1 \clesse \dots \clesse x_n$ if starting from $x_1\in \TT$ and
moving in the counterclockwise direction (that is, according to the positive orientation of the circle),
we pass the points in the order of their listing until after a full rotation we arrive at the initial point $x_1$.
Similarly for the strict precedence notation $x_1 \cless \dots \cless x_n $.
Correspondingly, arcs are defined as the set of points between two endpoints: $[a,b]:=\{ x\in \TT~:~ a\clesse x \clesse b\}$ etc.
Note that for $n=2$ $x_1\clesse x_2$ and also $x_2\clesse x_1$ hold simultaneously for all points $ x_1, x_2 \in \TT$, and that cyclic ordering of $n$ nodes is possible in $(n-1)!$ different ways.

The \enquote{large simplex on the torus} and its closure is defined as
\begin{align}
\LS:=&\ \{(y_1,y_2,\ldots,y_n)\in\bT^n:
\ y_1\cless y_2\cless \ldots \cless y_n (\cless y_1)
\},
\label{def:LS}
\\
\oLS:=&\ \{(y_1,y_2,\ldots,y_n)\in\bT^n:
\ y_1\clesse y_2\clesse \ldots \clesse y_n (\clesse y_1)
\}.
\label{def:oLS}
\end{align}
One may consider various permutations for the cyclic ordering, but by relabeling we will always assume that the ordering is just this natural cyclic order.

Pulling back by any $\pi_c^{-1}$
coordinatewise
we see that
\begin{equation}
\label{eq:LSSsigma}
\pi_c^{-1}(\oLS)
= \cup_{j=0}^{n-1}  \oSi^{(\sigma_j)},
\quad
\si_j(\ell):=\begin{cases}  \ell+j,
&\text{ if }
\ell=1,\ldots, n-j
\\ \ell+j-n, &
\text{ if } \ell=n-j+1,\ldots,n. \end{cases}
\end{equation}
This decomposition is not disjoint, given that $\oLS$ is connected.
Similarly, no representation of $\LS$ by
disjoint $\Si^{(\si)}$ exists,
as $\LS$ is connected, too.

For $\xx\in\oSi^{(\si)}$,
we have defined the intervals
\begin{align*}
I_0(\xx):=&\ [0,x_{\si(1)}], \quad I_j(\xx):=[x_{\si(j)},x_{{\si(j+1)}}]
\quad (1\le j \le n-1), \quad I_n(\xx):=[x_{\si(n)},1],
\end{align*}
and the corresponding
\enquote{interval maxima}
(in fact, supremums, not maximums) as
\begin{equation*}
m_j(\xx):=\sup\{F(\xx,t):\  t\in I_j(\xx)\},
\quad
j=0,1,\ldots,n.
\end{equation*}

Analogously, for an arbitrary $\yy\in \oLS$ we define
\begin{align}
I_j^*(\yy):=& \  \{t\in \bT:\
y_j\clesse t\clesse y_{j+1}  \},
\  j=1,2,\ldots,n-1,
\label{def:ijstar-1}
\\
I_n^*(\yy):=& \  \{t\in \bT:\   y_n\clesse t\clesse y_{1}
\},
\label{def:ijstar-2}
\end{align}
and the corresponding
\enquote{arc maximums}
\begin{equation*}
m_j^*(\yy):= \sup\{F(\yy,t):\ t\in I_j^*(\yy)\}, \quad j=1,2,\ldots,n.
\end{equation*}

The correspondence between $I_k^*(\yy)$ and $I_j(\xx)$, and also between
$m_k^*(\yy)$ and $m_j(\xx)$, can be described easily.
Fix $c\in\bT$ arbitrarily
and let $\yy\in\oLS$.
Then $\pi_c^{-1}(y_1),\ldots,\pi_c^{-1}(y_n)\in[0,1)$, moreover,
setting $x_j:=\pi_c^{-1}(y_j)$, the coordinates of the vector $\xx=(x_1,\ldots,x_n)\in\bR^n$ follow according to the cyclic ordering of the coordinates of $\yy$,
that is, if $y_j \cless c \clesse y_{j+1}$,
then $\xx \in\oSi^{(\si_j)}$.
Extending the definition of $\pi_c$ and its inverse to vectors,
we may write $\pi_c^{-1}(\yy)=\xx\in\oSi^{(\si_j)}$.
Then we find
$\pi_c(I_k(\xx))=\pi_c([x_{\si_j(k)},x_{\si_j(k+1)}])
=I_k^*(\yy)$
for $k=1,\ldots,n-1$,
while $\pi_c( I_0(\xx))\cup I_n(\xx))= I_n^*(\yy)$.
Accordingly, the interval and arc maxima correspond to each other as follows.
\begin{equation}
\label{eq.arcmaxintmax}
m^*_k(\yy)=m_k(\xx) \quad (1\le k \le n-1),\qquad m_n^*(\yy) =\max(m_0(\xx),m_n(\xx)).
\end{equation}
Note that this representation does depend
on the ordering of the $x_k$,
because $I_0(\xx)=[0,x_{\si_j(1)}]=[0,\pi_c^{-1}(y_{j+1})]=\pi_c^{-1}([c,y_{j+1}])$,
and $I_n(\xx)=[x_{\si_j(n)},1] = [\pi_c^{-1}(y_j),1]=\pi_c^{-1}([y_j,c])$.

In \cite{Sbornik, JMAA, Ural}
we have already investigated
the following minimax and maximin problems
on the interval $[0,1]$.
\begin{align*}
& \mol(\xx):= \
\max_{j=0,1,\ldots,n} m_j(\xx)
=\sup\{F(\xx,r):\ r\in[0,1]\},
\quad
\mul(\xx):=
\min_{j=0,1,\ldots,n} m_j(\xx),
\\
& M(\oSi):=
\inf\{\mol(\xx):\ \xx\in \oSi\},
\qquad
m(\oSi):=
\sup\{\mul(\xx):\ \xx\in \oSi\}.
\end{align*}
The analogous quantities on the torus are
\begin{align*}
& \mol^*(\yy):= \
\max_{j=1,\ldots,n} m_j^*(\yy)
=\sup\{F(\yy,t):\ t\in\bT\},
\quad
\mul^*(\yy):=
\min_{j=1,\ldots,n} m_j^*(\yy),
\\
& M^*( \oLS):=
\inf\{\mol^*(\yy):\ \yy\in \oLS\},
\quad
m^*( \oLS):= \
\sup\{\mul^*(\yy):\ \yy\in \oLS\}.
\end{align*}

Note that
$m_n^*(\pi_c(\xx))=\max\left(m_n(\xx),m_0(\xx)\right)$,
\begin{equation}
\label{eq:molstarvsmol}
\mol^*(\pi_c(\xx))=\mol(\xx),
\end{equation}
and in view of \eqref{eq:LSSsigma} we also have\footnote{In fact, this formula provides an alternative way to prove Theorem \ref{thm:minimax},
but not of Theorem \ref{thm:maximin},
so we have decided to follow the forthcoming approach.}
$M^*(\oLS)=\min_{j=1,\ldots,n} M(\oS^{(\si_j)})$.
However, there is no similar easy formula for $m^*(\oLS)$.

\section{Continuity results}

In this section, we collect the continuity properties
of maximum functions.
First, we recall Lemma 3.3 from \cite{JMAA}:
\begin{lemma}
\label{lem:mbarcont}
Let $n\in \NN$, $\nu_j>0$ ($j=1,\ldots,n$),
let $J$ be an $n$-field function on $[0,1]$, and
let $K$ be a kernel function on $[-1,1]$.

Then $\mol : [0,1]^n \to \RR$ is continuous.
\end{lemma}

From this we easily deduce the following.
\begin{proposition}
\label{prop:molstarcontinuous}
Let $n\in \NN$, $\nu_j>0$ ($j=1,\ldots,n$),
let $J$ be a field function and
let $K$ be a kernel function.

Then $\mol^* : \bT^n \to \RR$ is continuous.
\end{proposition}
\begin{proof}
Let $\avec\in \bT^n$,
$\avec=(a_1,\ldots,a_n)$
be fixed
and $c\in\bT\setminus\{a_1,\ldots,a_n\}$.
We show that $\mol^*$ is continuous in a
small neighborhood of $\avec$.
Let $\delta_0<\min_{j=1,\ldots,n}
\mathrm{dist}_\bT\left(c,a_j\right)$.
We pull back $\yy$ to $[0,1]^n$ coordinatewise:
$x_j:=\pi_c^{-1}(y_j)$, $j=1,\ldots,n$.
Then $x_j\in(0,1)$ and $x_j=x_j(y_j)$
is continuous
(since $\mathrm{dist}_\bT\left(y_j,a_j\right)<\delta_0$,
and $\pi_c^{-1}$ is continuous save at $c$),
so $\xx:=\xx(\yy):=\pi_c^{-1}(\yy)$ is changing continuously
in the given neighborhood of $\avec$.
We may write that $F(\yy,t)=F(\pi_c(\xx(\yy)),\pi_c(r))$
where $\pi_c(r)=t$
and after simplifying the notation,
we simply write $F(\yy,t)=F(\xx(\yy),r)$,
with $t\in \TT $ corresponding to $r:=\pi_c^{-1}(t) \in [0,1)$.

Therefore, with \eqref{eq:molstarvsmol},
we see that $\mol^*(\yy)=\mol(\xx(\yy))$.

The continuity of $\mol^*$
at $\avec$
follows from the continuity of
$\xx(\yy)$ at $\avec$ and the continuity of
$\mol$ at $\xx$,
the latter coming from Lemma \ref{lem:mbarcont}.
\end{proof}

We show continuity of the
arc maxima functions $m_j^*$
in some important cases.

\begin{proposition}
\label{prop:mjstarcont}
Let $n \in \NN$ and $k\in\{1,2,\ldots,n\}$ be fixed
and let $K_1,\ldots,K_n$ be arbitrary kernel functions.
\begin{enumerate}[label={(\alph*)}]
\item
\label{mjstarcont:parta}
Suppose that $J$ is an arbitrary $n$-field
function and all $K_j$,
$j=1,2,\ldots,n$ satisfy \eqref{cond:infty}.
Then  $m_k^*$ is extended
continuous on $\oLS$.
\item
\label{mjstarcont:partb}
Suppose that $J$ is an extended continuous field
function.
Then  $m_k^*$ is extended continuous on $\oLS$.
\item
\label{mjstarcont:partc}
If $J$ is an upper semicontinuous $n$-field function,
then
$m_k^*$ is upper semicontinuous on $\oLS$.
\end{enumerate}
\end{proposition}
\begin{proof}
To see \ref{mjstarcont:parta},
let $\avec\in \oLS$
be fixed.
If $I_k^*(\avec)\ne \bT$, then
let $c\in \bT\setminus I_k^*(\avec)$,
$c\not\in\{a_1,\ldots,a_n\}$.
Then there is a $j\in\{0,1,\ldots,n-1\}$
(see \eqref{eq:LSSsigma})
such that $\pi_c^{-1}(\avec) \in \oSi^{(\sigma_j)}$.
Moreover, $0<\pi_c^{-1}(a_1),\ldots,\pi_c^{-1}(a_n)<1$.

So, with
\eqref{eq.arcmaxintmax}
we can write $m_k^*(\yy)=m_k(\xx)$
when $\yy\in \oLS$ is close to $\avec$
and $\xx=\pi_c^{-1}(\yy)$, $\xx\in \oSi^{(\sigma_j)}$.
Since $\pi_c^{-1}:\oLS \rightarrow \oSi$ is
continuous near $\avec$,
and by Lemma 3.1
from \cite{JMAA},
$m_j:\oSi\rightarrow\uR$
is continuous,
we obtain the assertion of this
part.

If $I_k^*(\avec)=\bT$, then
we follow the same steps
with $c\in \bT\setminus \{a_k\}$, but
the arc $I_k^*(\avec)$
--and all arcs $I_k^*(\yy)$ with $\yy$ close to $\avec$--
necessarily split into two intervals via
$\pi_c^{-1}(\cdot)$,
so we use the second half of \eqref{eq.arcmaxintmax},
and we get
$m_k^*(\yy)=\max\big(
m_0(\xx),m_n(\xx)
\big)$
when $\yy\in \oLS$ is close to $\avec$.
Continuing with the same steps,
we obtain the assertion.

The proof of \ref{mjstarcont:partb}
is straightforward.

The proof of
\ref{mjstarcont:partc}
follows the same steps as
that of \ref{mjstarcont:parta},
using
Proposition 3.6 (a) from \cite{JMAA} and
that the maximum of
two upper semicontinuous
function is again upper semicontinuous.
\end{proof}

\begin{example}
If $J$ is not continuous,
and $K_j$ does not satisfy \eqref{cond:infty}
then $m_j^*$ is not continuous on
$\oLS$.

To see this, take $y^*\in\bT$,
where $J$ is not continuous:
we may assume $y^*=1/2$.
Let $n=2$ and consider $\xx=(x,x)$ where $x \approx 1/2$.
Then $m_1^*(\xx)=\sup_{I_1^*(\xx)} F(\xx,\cdot)= F(\xx,x)=J(x)+2K(0)$
and $m_1^*((1/2,1/2))=F((1/2,1/2),1/2)=J(1/2)+2K(0)$.
Hence $m_1^*(\cdot)$ is not continuous
at $(1/2,1/2)$.

\end{example}

Let us remark the Berge proved a maximum theorem about partial maxima of bivariate
functions, see, e.g., \cite{Berge1963},
but that result is not applicable here
since his approach requires
bivariate continuity.
In our case, $J$ may be discontinuous, and continuity of $m_j$ or $\mol$ is thus nontrivial.

\section{Perturbation lemmas}

The first perturbation lemma
describes the behavior of sum of translates
functions when two nodes are pulled apart.
It appeared in several forms,
e.g.,
in \cite{TLMS2018} (see Lemma 11.5),
\cite{Rankin}, Lemma 10 on p.~1069,
or \cite{Sbornik}, Lemma 3.1.
A similar  form can be found in
\cite{Fenton}
(see around formula (15) too).

\begin{lemma}[\textbf{Perturbation lemma}]
\label{lem:widening}
Let $K$ be a kernel function
which is periodic \eqref{Kperiodic}.
Let $0\le \alpha<a<b<\beta\le 1$
and $p, q >0$.
Set
\begin{equation}
\label{eq:mudef}
\mu:=\frac{p(a-\alpha)}{q(\beta-b)}.
\end{equation}
\begin{enumerate}[label={(\alph*)}]
\item
\label{pertlem:parta}
If $\mu=1$, then
\begin{equation}
\label{eq:wideningarc}
pK(t-\alpha)+qK(t-\beta) \le pK(t-a)+qK(t-b).
\end{equation}
holds for every $t\in [0,\alpha] \cup [\beta,1]$.
\item
\label{pertlem:partb}
Additionally, if $K$ is strictly concave,
then
\eqref{eq:wideningarc}
holds with strict inequality.
\item
\label{pertlem:partc}
If $\mu=1$, then
\begin{equation}
\label{eq:shorteningarc}
pK(t-\alpha)+qK(t-\beta) \ge pK(t-a)+qK(t-b).
\end{equation}
holds for every $t\in [a,b]$.
\item
\label{pertlem:partd}
Additionally, if $K$ is strictly concave,
then
\eqref{eq:shorteningarc}
holds with strict inequality.
\end{enumerate}
\end{lemma}

\begin{lemma}[\textbf{Trivial Lemma}]
\label{lem:trivi-comprehensive}
Let $f, g, h:D \to \uR$ be functions on some Hausdorff topological space $D$ and
assume that
\begin{enumerate}[label={(\roman*)}]
\item
\label{trivlemma:parti}
either $f,g,h$ are all upper semicontinuous,
\item
\label{trivlemma:partii}
or $f, g$ are extended continuous and
$h$ is locally upper bounded, but otherwise arbitrary.
\end{enumerate}
Let $\emptyset\neq A \subseteq B \subseteq D$ be arbitrary.
Assume
\begin{equation}
\label{eq:trivia}
f(t)<g(t) \quad\text{ for all } t\in A.
\end{equation}
If $A\subseteq B$ is a compact set,
then
\begin{equation}
\label{eq:trivia-plus}
\sup_A (f+h) < \sup_B (g+h) \qquad
\text{ unless} \qquad
h\equiv-\infty \quad\text{on }\quad A.
\end{equation}
\end{lemma}
\begin{proof}
The straightforward proof of \ref{trivlemma:parti}
was given in \cite{Sbornik} as Lemma 3.2.
The proof of \ref{trivlemma:partii}
is similar, so we leave it to the reader.
\end{proof}

The following lemma is rather similar to Lemma 4.1. of \cite{Sbornik}.
However, there are several differences, too, in which the below version is stronger than the former version.
First, here we do not assume upper semicontinuity of the field function,
which was made possible by the observation that there is a version of the
Trivial Lemma which relaxes on that condition on $h$ (even if using a little more assumption regarding continuity of $f$ and $g$, which, on the other hand, are clearly available).
Second, we assume non-degeneracy $w_{i+1}>w_i$ only for indices $i\in \cI$, again a delicate novelty in the current version.
Third, we drop the condition that $K$ be monotone, an essentially necessary assumption for the interval $[0,1]$,
but, as is already told in the Introduction, not required in the periodic case.
In view of all these differences, as well as in regard of the slightly different setup of having only $n$ arcs defined
by the $n$ node points (and not $n+1$ intervals),
we will present the full proof of the Lemma,
even if its basic idea and a large part of the details are repeating the former argument.

\begin{lemma}[\textbf{General maximum perturbation lemma on the torus}]
\label{lem:gen-max-perturbation}
Let $n\in\NN$ be a natural number, and
let $\nu_1,\dots, \nu_n>0$
be given positive coefficients.
Let $K$ be a
kernel function on $\TT$, and let
$J$ be an
arbitrary $n$-field function.

Let $\ww \in \oLS$  and $\cI \cup \cJ =\{1,\ldots,n\}$
be a non-trivial partition, and assume that for all $i \in \cI$,
we have $w_i <w_{i+1}$
(which holds in particular, independently of $\cI$, if
$\ww\in \LS$). 

Then, arbitrarily close to $\ww$,
there exists
$\ww' \in \oLS 
\setminus \{\ww\}$,
essentially different from and
less degenerate than $\ww$ in the sense that
\begin{equation}
\label{eq:essdiff}
w'_\ell\ne w_\ell \quad
\text{ unless } \quad
\{\ell-1,\ell\} \subset \cI  \quad
\text{ or } \quad  \{\ell-1,\ell\}
\subset \cJ
\end{equation}
and\footnote{Note that here $w_{\ell}=w_{\ell+1}$ excludes $\ell \in \cI$, so only $\cJ$ can contain all three indices listed.}
\begin{equation}
\label{eq:lessdeg}
w'_{\ell}\ne w'_{\ell+1} \quad
\text{unless} \quad
\{\ell-1,\ell,\ell+1\} \subset \cJ
\quad \text{and} \quad w_\ell=w_{\ell+1},
\end{equation}
(in particular, if
$\ww \in \LS$ 
then necessarily
$\ww' \in \LS$), 
and such that it satisfies
\begin{align}
\label{eq:genpertlemma-function-I}
&F(\ww',t)\le F(\ww,t)
\text{ for all }  t \in I_i^*(\ww')
\quad  \text{and} \quad
I_i^*(\ww')\subseteq I_i^*(\ww)
 \text{ for all }  i\in\cI;
\\
\label{eq:genpertlemma-function-J}
& F(\ww',t)\ge F(\ww,t)
\text{ for all } t \in I_j^*(\ww)
\quad \text{and} \quad
I_j^*(\ww')\supseteq I_j^*(\ww)
\text{ for all }  j\in\cJ.
\end{align}
As a result, we also have
\begin{equation}
\label{eq:genpertlemma-max}
m_i^*(\ww')\le m_i^*(\ww)
\text{ for }i\in\cI
\quad \text{and} \quad
m_j^*(\ww')\ge m_j^*(\ww) \text{ for }j\in\cJ
\end{equation}
for the corresponding torus maxima.

Moreover, if $K$ is strictly concave,
then the inequalities in
\eqref{eq:genpertlemma-function-I} and
\eqref{eq:genpertlemma-function-J}
are strict for all points in the respective
arcs where $J(t)\ne - \infty$.

Furthermore, for strictly concave $K$
the inequalities in \eqref{eq:genpertlemma-max}
are also strict for all indices $k$
with non-singular $I_k^*(\ww)$.
\end{lemma}

\begin{proof}
Before the main argument,
we observe that the assertion
in \eqref{eq:genpertlemma-max} is
indeed a trivial consequence of the previous inequalities
\eqref{eq:genpertlemma-function-J}
and \eqref{eq:genpertlemma-function-I},
so we need not give a separate proof for that.

Second,
the inequalities
\eqref{eq:genpertlemma-function-I} and
\eqref{eq:genpertlemma-function-J} follow from
\begin{align}
\label{eq:genpertlemma-fI}
&f(\ww',t)\le f(\ww,t)
\ (\forall t \in I^*_i(\ww'))
\quad \text{and} \quad
I^*_i(\ww')\subseteq I^*_i(\ww)
\quad \text{for all } i\in\cI;
\\
\label{eq:genpertlemma-fJ}
& f(\ww',t)\ge f(\ww,t) \
(\forall t \in I^*_j(\ww))
\quad \text{and} \quad
I^*_j(\ww')\supseteq I^*_j(\ww)
\quad \text{for all}\  j\in\cJ.
\end{align}
Moreover, strict inequalities
for all points $t$ with
$J(t)\ne -\infty$ will follow
from \eqref{eq:genpertlemma-function-I}
and \eqref{eq:genpertlemma-function-J}
if we can prove strict inequalities in
\eqref{eq:genpertlemma-fI} and \eqref{eq:genpertlemma-fJ}
\emph{for all values of $t$} in the said compact arcs.

Furthermore, in case we have strict inequalities in \eqref{eq:genpertlemma-fI} and \eqref{eq:genpertlemma-fJ} for all points $t$,
then for non-singular $I_k^*(\ww)$
this entails strict inequalities
also in \eqref{eq:genpertlemma-max}
(for the corresponding $k$).
To see this, one may refer back to the
Trivial Lemma \ref{lem:trivi-comprehensive}
\ref{trivlemma:partii}
with $\{f,g\}=\{f(\ww,\cdot),f(\ww',\cdot)\}$,
$h=J$, $\{A,B\}=\{I^*_k(\ww),I^*_k(\ww')\}$.
Here in the case when $k=i \in \cI$
we need to use that the arc $I^*_i(\ww)$ is not degenerate for $i \in \cI$,
hence if it is nonsingular, too,
then either $J|_{I^*_i(\ww')} \equiv -\infty$ and then
we have the strict inequality $m^*_i(\ww)>-\infty=m_i^*(\ww')$,
or $J(t)>-\infty$ at some points of $I^*_i(\ww')$,
and then using this Lemma furnishes the required strict inequality.
(The other case with $k=j \in \cJ$ is easier
because nonsingularity of $I_j^*(\ww)$ implies nonsingularity of the
larger arc $I_j^*(\ww')$, too, and then
the application of the Lemma need not be coupled by considerations of an identically $-\infty$ field.)

So the proof hinges upon showing \eqref{eq:essdiff}, \eqref{eq:lessdeg}, \eqref{eq:genpertlemma-fI} and
\eqref{eq:genpertlemma-fJ},
for any $n$-field function and any kernel function,
coupled with the strict inequality assertion in \eqref{eq:genpertlemma-fI} and
\eqref{eq:genpertlemma-fJ} for all $t$ belonging to the said compact arcs, in case $K$ is strictly concave.

For $n=0$ or $n=1$ there is no nontrivial partition of the index set $\{1,\ldots,n\}$,
hence the assertion is void and true.

For $n=2$ there is essentially only one way to split the index set in a nontrivial way,
so the statement will be part of the following,
more general setup in Case 0, which we prove directly.
Actually, the $n=2$ case can be proved directly from Lemma \ref{lem:widening},
but we will need the more general
Case 0 anyway.

\textbf{Case 0}.
We prove directly the assertion when $\cI, \cJ$ contain no neighboring indices,
so $n$ must be even and $\cI$ and $\cJ$ partition
$\{1,\ldots,n\}$ into the subsets of odd and even natural numbers from $1$ to $n$.

We can assume that $\cI=(2\NN+1) \cap \{1,\ldots,n\}$
and $\cJ=2\NN \cap \{1,\ldots,n\}$
(the other case being a simple change of the cut,
i.e., starting of the listing of the cyclic ordering of nodes from one node later) .

Note that whenever $w_j\in \cJ$ (i.e., when $j$ is even),
then we necessarily have $j-1, j+1\in \cI$,
and $w_{j-1}<w_j\le w_{j+1}<w_{j+2}$.

Denote $\de:=\min_{i \in \cI}|I^*_i(\ww)|$,
(where $|I^*_i(\ww)|$ is the length of the arc $[w_i,w_{i+1}]$)
which is positive by condition.

Our new perturbed node system $\ww'$
will be, with an arbitrary
$0<h<\frac{1}{2} \delta/\max\{\nu_1,\ldots, \nu_n\}$,
the system
\begin{equation}
\label{eq:wprimeoddeven}
\ww':=(w_1',\ldots,w_n') \quad
\text{with} \quad
w'_\ell:=w_\ell-(-1)^\ell \frac{1}{\nu_\ell} h,
\quad \ell=1,2,\ldots,n.
\end{equation}
The definition guarantees, that $w'_\ell=w_\ell$ happens for
no index $\ell$, furnishing \eqref{eq:essdiff}.

It is easy to see that by the choice of the perturbation lengths,
no two consecutive nodes will change ordering or reach each other:
for $j\in \cJ$ $w_j, w_{j+1}$ are changed to become farther away,
while for $i\in \cI$ $w_i, w_{i+1}$ are moved closer,
but only by $\nu_i h + \nu_{i+1} h < \de<w_{i+1}-w_i$.
It follows that the ordering of nodes remains in
$\oLS$.
(In a minute we will see that,
moreover, there remains no degenerate arc,
so
$\ww' \in \LS\setminus\{\ww\}$.)

Obviously, $I^*_j(\ww') \supset I^*_j(\ww)$ holds
for all $j\in \cJ$, and
$I^*_i(\ww') \subset I^*_i(\ww)$ holds
for all $i\in \cI$,
with the inclusions strict, furnishing the second parts of
\eqref{eq:genpertlemma-function-I} and \eqref{eq:genpertlemma-function-J}
(matching the second parts of \eqref{eq:genpertlemma-fI}
and \eqref{eq:genpertlemma-fJ}, too).
In particular, even if $I^*_j(\ww)$ may
be
degenerate
for some $j\in \cJ$, i.e.,
for some even $j$, the $j$th arc of the perturbed system will not be such:
$w'_j<w_j\le w_{j+1}<w'_{j+1}$ for any even $j \in \cJ$.
We get $\ww' \in \LS$,
entailing \eqref{eq:lessdeg}.

Take now an even indexed arc
$I^*_{2k}(\ww)=[w_{2k},w_{2k+1}]$,
so $2k \in \cJ$.
(When $2k=n$, then we must read $2k+1=n+1\equiv 1$, i.e. $w_{2k+1}=w_1$.)
Our perturbation of nodes in \eqref{eq:wprimeoddeven} can now be
grouped as \emph{pairs of changing nodes} $w_{2\ell-1},w_{2\ell}$
among $w_1,\ldots w_{2k}$, and
then again among $w_{2k+1},\ldots, w_n$,
recalling that $n$ is even.
Now, \emph{the pairs} are always changed so that the arcs in between shrink,
and shrink exactly as is described in Lemma \ref{lem:widening}.
We apply this lemma for each pair of such nodes with the choices
$a=w'_{2\ell-1}$, $b=w'_{2\ell}$,
$\alpha=w_{2\ell-1}$, $\beta=w_{2\ell}$,
$p=\nu_{2\ell-1}$, $q=\nu_{2\ell}$.
This gives that for each such pair of changes,
for $t$ \emph{outside of the enclosed arc}
$(w_{2\ell-1},w_{2\ell})=I^*_{2\ell-1}(\ww)$ we have
\begin{equation}
\label{eq:genpert1}
\nu_{2\ell-1}K(t-w'_{2\ell-1})
+\nu_{2\ell}K(t-w'_{2\ell})
\geq
\nu_{2\ell-1}K(t-w_{2\ell-1})
+\nu_{2\ell}K(t-w_{2\ell}).
\end{equation}
Note that $I^*_{2k}(\ww)$,
hence any $t \in I^*_{2k}(\ww)$, is \emph{always outside} of the arcs $I^*_{2\ell-1}(\ww)$,
therefore \eqref{eq:genpert1} holds for the given,
fixed $t\in I^*_{2k}(\ww)$ and for all $\ell$.
So we find
\begin{align}
\notag
f(\ww,t)& = \sum_{\ell=1}^{n/2}
\left(\nu_{2\ell-1}K(t-w_{2\ell-1})
+\nu_{2\ell}K(t-w_{2\ell})\right)
\\&
\le \sum_{\ell=1}^{n/2}
\left(\nu_{2\ell-1}K(t-w_{2\ell-1}')
+\nu_{2\ell}K(t-w_{2\ell}')\right)
=f(\ww',t).
\end{align}
Furthermore, all the appearing inequalities are strict in case $K$ is strictly concave.
We have proved \eqref{eq:genpertlemma-fJ},
even with strict inequality under appropriate assumptions.

The proof of \eqref{eq:genpertlemma-fI} runs analogously
by grouping the change of nodes as a
change of pairs $w_{2\ell}, w_{2\ell+1}$
for $\ell=1,\dots,n/2$, writing $w_{2(n/2)+1}=w_{n+1}=w_1$ according to periodicity.
For these arcs $2\ell \in \cJ$ and the
arcs are getting larger after the perturbation,
so outside these enlarged arcs--that is, for all points
which belong to any $I^*_{2k-1}(\ww')$ for some fixed $k$--the changed value $f(\ww',t)$ will not exceed
(and in case of strict concavity, will be strictly smaller than) $f(\ww,t)$.
This means a nonincreasing (decreasing) change for all $I^*_i(\ww')$,
entailing \eqref{eq:genpertlemma-fI} together
with the  respective strict inequality statement.

The proof of Case 0 is thus completed.

\medskip

Therefore, we have also the case $n=2$ proved.
From here we continue our argumentation by induction.
Let now $n>2$ and assume, as inductive hypothesis,
the validity of the assertions
for all $n^* \le \widetilde{n}:=n-1$ and for any choice of kernel- and $n^*$-field functions.

In view of Case 0 above,
there remains the case when there are neighboring indices $k-1, k$ belonging to the same index set $\cI$ or $\cJ$.
In view of the cyclic ordering and to avoid indexing complications,
assume that we also have $1 <k <n$, which is a possibility for any $n$ at least $3$.
We separate two cases.

\medskip
\textbf{Case 1}. Assume first that
$w_{k-1}<w_k<w_{k+1}$ holds.

Then we consider the kernel function $\widetilde{K}:=K$,
and the $\widetilde{n}$-field function $\widetilde{J}:=\nu_k K(\cdot-w_k)$
(which is indeed an $\widetilde{n}$-field function because it attains $-\infty$ only at most at one point,
namely $w_k$, in case $K$ is singular.)

Correspondingly, now the sum of translates function $\widetilde F$ is formed
by using $\widetilde{n}=n-1$ translates with coefficients
$\nu_1,\ldots,\nu_{k-1},\nu_{k+1},\ldots,\nu_n$ and
with respect to the node system
\[
\widetilde{\ww}:=
(w_1,w_2,\ldots,w_{k-1},w_{k+1},\ldots,w_n).
\]
Formally, the indices change:
$\widetilde{w}_\ell=w_\ell$ for $\ell=1,\ldots,k-1$,
but $\widetilde{w}_\ell=w_{\ell+1}$ for $\ell=k,\ldots,n-1$, the $k$th coordinate being left out.

We apply the same change of indices in the partition:
$k$ is dropped out (but the corresponding index set
$\cI$ or $\cJ$ will not become empty, for it contains $k-1$);
and then shift indices one left for $\ell>k$: so
\begin{align*}
\widetilde{\cI}&:=
\{i \in \cI: i<k\} \cup
\{ i-1 : i \in \cI, i> k \}
\text{  and} \\
\widetilde{\cJ}&:=\{j \in \cJ: j<k\}
\cup \{j-1 : j\in \cJ, j>k\}.
\end{align*}
Observe that $\widetilde{F}(\widetilde{\ww},t)=f(\ww,t)$
for all $t\in \TT$, while
\[
I^*_\ell(\widetilde{\ww})=
\begin{cases}
I^*_\ell(\ww), \quad &
  \text{if} \quad \ell <k-1, \\
I^*_{k-1}(\ww)\cup I^*_{k}(\ww), \quad &
  \text{if} \quad \ell =k-1, \\
I^*_{\ell+1}(\ww), \quad &
  \text{if} \quad \ell \ge k.
\end{cases}
\]
(Here we make a little use of the choice that $1<k<n$,
so we need not bother too much with the cyclic renumbering etc.)

Note that $I^*_i(\widetilde{\ww})=[\widetilde{w}_i,\widetilde{w}_{i+1}]$ is still nondegenerate
whenever $i \in \widetilde{\cI}$,
for the arc is either a former arc belonging to some $i\in \cI$,
or the union of two such arcs.
Also, ordering of nodes is kept intact,
so  $\widetilde{\ww} \in \oLStilden$
(where $\oLSn$ and
$\oLStilden$
denote the cyclic simplices of the
corresponding dimension).

Now we apply the inductive hypothesis for the new configuration.
This yields a perturbed node system $\widetilde{\ww}' \in \oLStilden \setminus \{\widetilde{\ww}\}$,
arbitrarily close to $\widetilde{\ww}$,
with the asserted properties.
It is important that here the
ordering of the nodes remain the order fixed
in $\oLStilden$,
so if $\widetilde{\ww}'$ was closer to $\widetilde{\ww}$
than the distance $\delta$
of $w_k$ from $\{w_{k-1},w_{k+1}\}$,
then the $n$-term node system $\ww'$, obtained by
keeping the nodes from $\widetilde{\ww}'$
and inserting back $w_k$ to the $k$th place
(and shifting the following indices by one)
will again be ordered as $\ww$ was,
i.e., $\ww' \in \oLSn$
(and of course $\ne \ww'$,
as already $\widetilde{\ww}'\ne \widetilde{\ww}$).
Moreover, $w'_k=w_k$ is still not equal
to any of the nodes $w'_{k-1}, w'_{k+1}$,
because $\mathrm{dist}_{\bT^n}(\ww',\ww)<\delta$.

We now set to prove \eqref{eq:essdiff}.
First, if for some $\ell<k$
we have $w_\ell'=w_\ell$,
then $\widetilde{w}'_\ell=:w_\ell'=w_\ell =:\widetilde{w}_\ell$,
and then by the inductive hypothesis $\{\ell-1,\ell\}\subset \widetilde{\cI}$
or $\widetilde{\cJ}$.
This gives \eqref{eq:essdiff} in case $\ell<k$,
because below $k$ the partition sets $\cI$ and $\widetilde{\cI}$ (and $\cJ$ and $\widetilde{\cJ}$, respectively), consist of the same indices.

Take now some $\ell>k$ with $w_\ell'=w_\ell$.
Then $\widetilde{w}'_{\ell-1}=:w_\ell'=w_\ell=:\widetilde{w}_{\ell-1}$
and
$\{\ell-2,\ell-1\}\subset \widetilde{\cI}$
or $\widetilde{\cJ}$
by the inductive hypothesis. If it was $\ell>k+1$, too, then
this means $\{\ell-1,\ell\}\subset \cI$
or $\cJ$, that is, \eqref{eq:essdiff}.

If, however, $\ell=k+1$, then $\ell-2=k-1$ and $\ell-1=k$,
and by construction we get that $\{k-1, k+1\} \subset \cI$ or $\cJ$.
Given that we already have from the outset that $k$ belongs to the same index set as $k-1$, this altogether gives $\{k-1, k, k+1\} \subset \cI$ or $\cJ$,
which is more than needed for \eqref{eq:essdiff}.

Finally, if $\ell=k$,
then $w'_k=w_k$ by construction,
but then we had by assumption that $k-1$
and $k$ belonged to the same set $\cI$
or $\cJ$, so that \eqref{eq:essdiff} is satisfied.

Consider now the assertions of \eqref{eq:lessdeg}.
As above, there is no problem with respective index sets all
remaining the same or all being shifted by one,
i.e., if either $\ell<k-1$ or if $\ell>k+1$.

Recall that we assumed $w_{k-1}<w_k<w_{k+1}$ at the outset, and chose the perturbation small enough to keep this strict ordering.
Therefore, $w'_{k-1}=w'_k$ and $w'_k=w'_{k+1}$ are excluded, and
only the case of $\ell=k+1$ and $w'_{k+1}=w'_{k+2}$ remains to be dealt with.
Now, $\widetilde{w}'_k=:w'_{k+1}=w'_{k+2}:=\widetilde{w}'_{k+1}$,
so  $\{k-1,k,k+1\} \subset \widetilde{\cJ}$ and $\widetilde{w}_k:=\widetilde{w}_{k+1}$
according to the inductive hypothesis.
The latter means $w_{k+1}=w_{k+2}$, while for the indices
we obtain $\{k-1,k+1,k+2\}\subset \cJ $.
But $k-1$ and $k$ belong to the same index set, so that also $k$ must belong to $\cJ$,
and therefore $\{k-1,k,k+1,k+2\}\subset \cJ$, entailing \eqref{eq:lessdeg}.

\medskip
Using that $k-1$ and $k$ belong to the same index set $\cI$ or $\cJ$,
it is easy to check that $I^*_i(\widetilde{\ww}') \subseteq I^*_i(\widetilde{\ww})$
for all $i\in \widetilde{\cI}$ is equivalent to $I^*_i(\ww')\subseteq I^*_i(\ww)$
for all $i\in {\cI}$, and $I^*_j(\widetilde{\ww}') \supseteq I^*_j(\widetilde{\ww})$
for all $j\in \widetilde{\cJ}$ is equivalent to $I^*_j(\ww')\supseteq I^*_j(\ww)$ for all $j\in {\cJ}$.
Further, the assertions \eqref{eq:genpertlemma-function-I}, \eqref{eq:genpertlemma-function-J}
from the inductive hypotheses lead to the assertions \eqref{eq:genpertlemma-fI}, \eqref{eq:genpertlemma-fJ}
for the original case.
Therefore,
by the preliminary observations also \eqref{eq:genpertlemma-function-I}, \eqref{eq:genpertlemma-function-J} follow.
Moreover, the assertion regarding strict inequalities for all $t$
in case of a strictly concave $K$ follow from the respective strict inequalities
for the inductive hypotheses, noting that
$I_{k-1}^*(\widetilde{\ww})=I_{k-1}^*(\ww)\cup I^*_{k}(\ww)$
can handle the necessary inequalities for both indices $k-1$ and $k$,
because these belong to the same index set $\cI$ or $\cJ$, and
hence invoke inequalities in the same direction.

\medskip
\textbf{Case 2}. Consider now the case when some of the partition sets $\cI, \cJ$
contain some neighboring indices $k-1, k$, such that $w_{k-1}=w_k$
or $w_k=w_{k+1}$ holds, too.
Then this index set cannot be $\cI$,
for indices in $\cI$ the respective arcs were supposed to be nondegenerate.
So, $k-1, k \in \cJ$.

Repeating the above argument in Case 1 then works.
Let us detail, why.

The main point where we needed that $w_{k-1}<w_k<w_{k+1}$ was where
we wanted to see that the new node system $\widetilde{\ww}'$,
provided by the induction hypothesis,
not only preserves cyclic ordering of nodes from $\widetilde{\ww}$,
but even with re-inserting $w_k$ and thus manufacturing $\ww'$ will still result
in a point belonging to $\oLSn$.

Observe that the inductive hypothesis,
in view of $k-1, k \in \cJ$,
furnishes that $I^*_{k-1}(\widetilde{\ww})$ is subject to growth,
so it will still contain the point $w_k$, that is,
$w'_{k-1}:=\widetilde{w}'_{k-1}\le \widetilde{w}_{k-1}:=w_{k-1} \le w_k \le w_{k+1}=:\widetilde{w}_k\le \widetilde{w}'_{k} =:w'_{k+1}$.
It follows immediately that we will thus have $\ww' \in \oLS$ at least.

Moreover, as above, equality of other perturbed and original nodes $w'_\ell$
and $w_\ell$ ($\ell \ne k$) can occur only when they occurred also in $\widetilde{\ww}$, hence in $\ww$, too.
Checking the arising conditions for the indices can be done mutatis mutandis the above case,
proving \eqref{eq:essdiff}, on noting that for $\ell=k$ we already have $k-1,k\in \cJ$ by assumption.

To prove \eqref{eq:lessdeg}, assume now the identity $w'_\ell=w'_{\ell+1}$.
Again, we separate cases according to the size of $\ell$ and start with the case $\ell<k-1$.
This implies
$\widetilde{w}'_\ell=:w'_\ell=w'_{\ell+1}:=\widetilde{w}'_{\ell+1}$,
hence also $\widetilde{w}_\ell=\widetilde{w}_{\ell+1}$ and $\ell-1,\ell,\ell+1 \in \widetilde{\cJ}$
in view of the inductive hypothesis, so  $w_\ell=w_{\ell+1}$ and $\ell-1,\ell,\ell+1 \in \cJ$, too.
Similarly, if $\ell>k$ then the identity $w'_\ell=w'_{\ell+1}$ implies
$\widetilde{w}'_{\ell-1}=:w'_\ell=w'_{\ell+1}:=\widetilde{w}'_{\ell}$,
hence also $\widetilde{w}_{\ell-1}=\widetilde{w}_\ell$ and $\ell-2,\ell-1,\ell \in \widetilde{\cJ}$
in view of the inductive hypothesis, so  $w_\ell=w_{\ell+1}$ and
if $\ell>k+1$, then $\ell-1,\ell,\ell+1 \in \cJ$, too,
while if $\ell=k+1$, then we get only $k-1, k+1,k+2 \in \cJ$,
but then again we remind to $k\in \cJ$ and get $\{k-1,k,k+1,k+2\}\subset \cJ$, proving \eqref{eq:lessdeg}.

It remains to deal with $\ell=k-1$ and $\ell=k$, which did not occur in Case 1 above.
Now, $\ell=k-1$ means that we have the identity $w'_{k-1}=w'_k$.
Since $k-1,k\in \cJ$ by assumption, we also have $k-1 \in \widetilde{\cJ}$,
hence $I^*_{k-1}(\widetilde{\ww})$ cannot shrink, and therefore $\widetilde{w}'_{k-1} \le \widetilde{w}_{k-1}$.
Moreover, if we had strict inequality here,
then we would have to have $w'_{k-1}:=\widetilde{w}'_{k-1} < \widetilde{w}_{k-1}:=w_{k-1}\le w_k=:w'_k$,
although we supposed the contrary now.
So, we must have $\widetilde{w}'_{k-1} = \widetilde{w}_{k-1}$,
the inductive hypothesis applies, and
we derive that both $k-2$ and $k-1$ belong to the same index set -- that is, because of $k-1$, to $\cJ$.
However, we already know by assumption also
$k\in \cJ$, so altogether $\{k-2,k-1,k\}\subset \cJ$, as needed.

The case $\ell=k$ is similar.
If $w'_k=w'_{k+1}$, then taking into account
that $I^*_{k-1}(\widetilde{\ww})$ cannot shrink (as $k-1 \in \widetilde{\cJ}$)
we must have $w'_{k+1}:=\widetilde{w}'_k \ge \widetilde{w}_k:=w_{k+1}\ge w_k=w'_k$ entailing that all inequalities are in fact equalities, and in particular both $w_{k+1}=w_k$ and $\widetilde{w}'_k = \widetilde{w}_k$.
Referring to the inductive hypothesis this furnishes $k-1,k \in \widetilde{\cJ}$,
that is, $k-1,k+1 \in \cJ$,
whilst $k\in \cJ$ by the original condition, altogether yielding $\{k-1,k,k+1\}\subset \cJ$, as needed.

So, we proved \eqref{eq:essdiff} and \eqref{eq:lessdeg} for this case, too.
The proof for the remaining inequalities and strictness of them in case of a strictly concave kernel is identical to the argument in Case 1.

We thus conclude the proof of Case 2,
whence the whole Lemma.

\end{proof}

\begin{rem}
Although the formulation of the Lemma is a bit complicated,
one may note that assuming $w_i <w_{i+1}$
for all $i \in \cI$ is absolutely natural and minimal.
Natural, because if we want to decrease the arcs $I_i^*(\ww)$ for all $i\in\cI$,
then these arcs must shrink,
hence we cannot perform this change
when they are already degenerate one point intervals.
Minimal, because we do not assume similar conditions
for any $I_j^*(\ww)$ with $j\in \cJ$,
so even the less that $\ww \in \LS$.
\end{rem}

\begin{rem} If we had $\ww \in \LS$ from the outset, then we will as well have $\ww' \in \LS$.
\end{rem}

\begin{cor}
If we only want respective inequalities for the
$m_j^*$, without requiring strict inclusions regarding the
underlying intervals, then we can apply a further
perturbation, now leading to $\ww'' \in \LS$
(the point being that with different endpoint nodes !) and
still satisfying the required strict inequalities between
the $m_j^*$,
provided that we had strict inequalities
(so, e.g., $K$ was strictly concave) and provided the $m_j^*$
change continuously.
\end{cor}

Note that this latter condition of continuity of the $m^*_j$ is satisfied
if $K$ is singular or
if $J$ is continuous, see Proposition \ref{prop:mjstarcont} \ref{mjstarcont:parta}
and \ref{mjstarcont:partc}.

\section{Minimax and maximin theorems}\label{sec:mnimaxandmaximin}

\subsection{Minimax for strictly concave kernels}

The following theorem contains, as a rather special case with the choice of the log-sine kernel $K(t):=\log|\sin (\pi t)|$, the above stated Theorem \ref{thm:Bojanov}.

\begin{thm}
Let $n \in \NN$ and $\nu_1,\ldots,\nu_n>0$,
let $K$ be a singular, strictly concave periodic kernel function, and
let $J$ be an arbitrary periodic $n$-field function.

Then there exists a minimax point $\ww$ on $\overline{\LS}$,
it belongs to the open
\enquote{cyclic simplex} $\LS$,
and it is an equioscillation point.
\label{thm:minimax}
\end{thm}

\begin{proof}
We already know that $\mol^*$ is continuous,
thus it attains its infimum at a minimum point
(where, in view of $\mol^* : \TT^n \to \RR$, a finite minimax value is attained).

Now assume for a contradiction that
the obtained minimax node system $\ww$ is not an equioscillating system.
Then there are indices with $m^*_i(\ww)=\mol^*(\ww)$,
but not all indices are such.
So, take $\cI:=\{ i~:~ m_i^*(\ww)=\mol^*(\ww)\}$
and $\cJ:=\{1,2,\ldots,n\}\setminus \cI=\{ j~:~ m_j^*(\ww)<\mol^*(\ww)\}$.
These index sets will define a nontrivial
partition of the full set of indices from 1 to $n$.

In order to apply the Perturbation Lemma we will need that for $i\in \cI$ the endpoint nodes are different: $w_i<w_{i+1}$.
Given that $m_i^*(\ww)=\mol^*(\ww)$, this is certainly so
in case $K$ is singular, for then $m_\ell^*(\ww)=-\infty<\mol^*(\ww)$ for any degenerate arc $I^*_\ell(\ww)$.
Let us now apply Lemma \ref{lem:gen-max-perturbation},
which results in a new node system $\ww'$,
admitting the same cyclic ordering and lying
arbitrarily close to $\ww$, and with $m^*_i(\ww')<\mol^*(\ww)$ for all
$i\in \cI$.
On the other hand, the $m_j^*(\ww')$ may exceed $m^*_j(\ww)$
for $j\in \cJ$,
but only by arbitrarily little, because
$\ww'$ is sufficiently close to $\ww$ and the
$m_j^*$'s are continuous on $\oLS$ in view of the
singularity of $K$,
see Proposition \ref{prop:mjstarcont} \ref{mjstarcont:parta}.
So, in all,
we will have $m^*_k(\ww')<\mol^*(\ww)$ for all $k=1,\ldots,n$,
hence $\mol^*(\ww')<\mol^*(\ww)$,
contradicting to minimality of $\ww$.

Noting that an equioscillation node system is
necessarily nondegenerate for singular kernels,
we conclude the proof.

\end{proof}

\subsection{Maximin for strictly concave kernels}

\begin{thm}\label{thm:maximin} Let $n \in \NN$ and $\nu_1,\ldots,\nu_n>0$,
let $K$ be a singular, strictly concave periodic kernel function, and
let $J$ be an arbitrary periodic $n$-field function.

Then there exists a maximin point $\zz$ on $\overline{\LS}$,
it belongs to the open
\enquote{cyclic simplex} $\LS$, and
it is an equioscillation point.
\end{thm}

\begin{proof}
Again, as all the $m_k^*$ are continuous,
so is their minimum.
The maximum of that minimum is finite,
because there are points with all $m_k^*$ finite:
e.g., take the above found minimax point $\ww$.
Given that $K$ is assumed to be singular,
that gives that such a point
with $\mul^*(\zz)>-\infty$ cannot be degenerate,
i.e., $\zz \in \LS$.
Also, by continuity of the $m_k^*$ and hence of $\mul^*$,
there exists a maximin point $\zz \in \LS$.

Assume for a contradiction that this point is not an equioscillation point.
Consider $\cI:=\{ i~:~ m_i^*(\zz)>\mul^*(\zz)\}$
and $\cJ := \{1,\dots,n\}\setminus \cI=\{j~:~ m_j^*(\zz)=\mul^*(\zz)\}$.
This is a nontrivial partition of the index set $\{1,\ldots,n\}$,
while $\zz \in \LS$,
hence the Perturbation Lemma \ref{lem:gen-max-perturbation} can be applied, and
we are led to a new system $\zz'\in \LS$,
with $m^*_j(\zz')>m_j^*(\zz)=\mul^*(\zz)$ for all $j\in \cJ$.
However, we also have $m_i^*(\zz')>m^*_i(\zz)-\ve >\mul^*(\zz)$ for all $i\in \cI$,
if $\ve$ was chosen small enough and $\zz'$ close enough to $\zz$,
because $m_i^*$ is continuous.
So, in all, we find $\mul^*(\zz')>\mul^*(\zz)$,
and $\zz$ could not be a maximin point.
The obtained contradiction proves the assertion.
\end{proof}

\subsection{Extension to concave kernel functions}

To extend Theorem \ref{thm:minimax}
to general concave kernels,
we apply limiting arguments similar to \cite{JMAA}, p.~18.

\begin{thm}
Let $n \in \NN$ and $\nu_1,\ldots,\nu_n>0$,
let $K$ be a singular periodic kernel function, and
let $J$ be an arbitrary periodic $n$-field function.

Then there exists a minimax point $\ww^*$ on $\overline{\LS}$,
it belongs to the open \enquote{cyclic simplex} $\LS$,
and it is an equioscillation point.
\label{thm:genminimax}
\end{thm}

\begin{proof}
Let $K$ be a singular,
$1$-periodic kernel function
which is not necessarily strictly concave.
Let $K^{(\eta)}(t):=K(t)+\eta |\sin\pi t|$
where $\eta>0$.
Then $K^{(\eta)}$ is a strictly concave kernel function,
which is also singular and $1$-periodic.
We will denote the corresponding
maximum functions and minimax quantity by
$m_j^*(\eta,\yy)$,
$\mol^*(\eta,\yy)$,
and $M^*(\eta,\oLS)$.
Taking into account that $K^{(\eta)}$ converges to $K$ uniformly, we find the same for $F(\eta,\cdot) \searrow  F$ and hence even
$\mol^*(\eta,\yy)\searrow \mol^*(\yy)$ uniformly.
Hence $M^*(\eta,\LS)\searrow M^*(\LS)$, too.

Let $\ee(\eta)\in \oLS$ be a node system
such that $\mol^*(\eta,\ee(\eta))=M^*(\eta,\oLS)$.
By Theorem \ref{thm:minimax},
$\ee(\eta)\in \LS$
and it is an equioscillating node system:
$m_1^*(\eta,\ee(\eta))=\ldots=m_n^*(\eta,\ee(\eta))$.
Since $\oLS$ is compact,
there exists $\eta_k\searrow 0$
such that $\ee(\eta_k)\rightarrow\ee$
for some $\ee\in\oLS$ as $k\rightarrow \infty$.

Then $\mol^*(\ee)=M^*(\oLS)$.
Indeed, $\mol^*(\ee)\ge M^*(\oLS)$ and for the other
direction, let $a>M^*(\oLS)$.
Then for all sufficiently large $k$ we have $a\geq M^*(\eta_k,\oLS)$,
so we can conclude
\[
a\geq M^*(\eta_k,\oLS)
=\mol^*(\eta_k,\ee(\eta_k))
\geq \mol^*(\ee(\eta_k))
\]
where we used $\mol^*(\eta,\cdot)\ge \mol^*(\cdot)$.
Letting $k\to\infty$ we conclude,
by the continuity of
$\mol^*$ and by $\ee(\eta_k)\to \ee$,
that
$a\geq \mol^*(\ee)$ and
$M^*(\oLS)\geq \mol^*(\ee)$ follows.
The claim is proved.

Next we claim that $\ee$ is an equioscillation point.
Indeed, assume for a contradiction that
for some $j\in \{1,\dots,n\}$
we have $m_j^*(\ee)<\mol^*(\ee)$.
Then there is $k_0\in \NN$ such that
$m_j^*(\eta_{k_0},\ee) <\mol^*(\ee)$.
Since $m_j^*(\eta_{k_0},\cdot)$ is continuous
(the kernel functions are singular; see Proposition \ref{prop:mjstarcont} \ref{mjstarcont:parta}),
there is $\delta>0$ such
that for every $\yy\in \oLS$
with $\mathrm{dist}_{\bT^n}(\yy,\ee)<\delta$
one has
$m_j^*(\eta_{k_0},\yy)
<\mol^*(\ee)$, too.

There is $n_0\in \NN$ such that for every $k\geq n_0$
we have
$\mathrm{dist}_{\bT^n}(\ee(\eta_{k}),\ee)<\delta$.
So for $k\geq\max\{k_0,n_0\}$
we can write
\[
m_j^*(\eta_{k},\ee(\eta_k))
\leq
m_j^*(\eta_{k_0},\ee(\eta_k))
<
\mol^*(\ee) = M^*(\oLS)
\leq
\mol^*(\ee(\eta_k))
\leq
\mol^*(\eta_k,\ee(\eta_k)).
\]
This is a contradiction,
since $m_i^*(\eta_{k},\ee(\eta_k))
=\mol^*(\eta_k,\ee(\eta_k))$
for each $i\in \{1,\dots,n\}$.

Therefore, $\ee$ is necessarily an equioscillation point.
As such, it cannot have any degenerate
subarcs, for any
degenerate subarc $I_k^*(\ee)=\{e_k\}$
would yield a
singular value  $m_k^*(\yy)=-\infty$
according to the
singularity of $K$. Hence $\ee \in \LS$.

Choosing $\ww^*:=\ee$ concludes the proof.
\end{proof}

Now we extend  Theorem \ref{thm:maximin}
to general concave kernels,
but this time we approximate
the kernel from below.

\begin{thm}
\label{thm:maximin-nonstrict}
Let $n \in \NN$ and $\nu_1,\ldots,\nu_n>0$,
let $K$ be a singular periodic kernel function, and
let $J$ be an arbitrary periodic $n$-field function.

Then there exists a maximin point $\zz^*$
on $\overline{\LS}$,
it belongs to the open \enquote{cyclic simplex} $\LS$,
and it is an equioscillation point.
\end{thm}

\begin{proof}
Let $K$ be a singular,
$1$-periodic kernel function
which is not necessarily strictly concave.
Let $K^{(\eta)}(t):=K(t)+\eta \big( |\sin\pi t|-1\big)$
where $\eta>0$.
Then $K^{(\eta)}$ is a strictly concave kernel function,
which is also singular and $1$-periodic
and $K^{(\eta)}\nearrow K$
as $\eta\searrow 0$,
moreover
this convergence is uniform.
Again,
we denote the corresponding
maximum functions and maximin quantity by
$m_j^*(\eta,\yy)$,
$\mul^*(\eta,\yy)$
and $m^*(\eta,\oLS)$.
Due to the uniform convergence of $K^{(\eta)}$'s,
we also have
$\mul^*(\eta,\yy)\nearrow \mul^*(\yy)$
uniformly, in the extended sense.
Therefore $m^*(\eta,\oLS)\nearrow m^*(\oLS)$, too.

Let $\ee(\eta)\in \oLS$ be a node system
such that $\mul^*(\eta,\ee(\eta))=m^*(\eta,\oLS)$.
By Theorem \ref{thm:maximin},
$\ee(\eta)\in \LS$
and it is an equioscillating node system:
$m_1^*(\eta,\ee(\eta))=\ldots=m_n^*(\eta,\ee(\eta))$.
Since $\oLS$ is compact,
there exists $\eta_k\searrow 0$
such that $\ee(\eta_k)\rightarrow\ee$
for some $\ee\in\oLS$ as $k\rightarrow \infty$.

Then $\mul^*(\ee)=m^*(\oLS)$.
Indeed, $\mul^*(\ee)\le m^*(\oLS)$ and for the other
direction, let $b<m^*(\oLS)$.
Then for all sufficiently large $k$
we have $b\leq m^*(\eta_k,\oLS)$,
so we can conclude
\[
b\leq m^*(\eta_k,\oLS)
=\mul^*(\eta_k,\ee(\eta_k))
\leq \mul^*(\ee(\eta_k))
\]
where we used $\mul^*(\eta,\cdot)\le \mul^*(\cdot)$.
Letting $k\to\infty$ we conclude,
by the extended continuity of
$\mul^*$  and by $\ee(\eta_k)\to \ee$,
that
$b\leq \mul^*(\ee)$ and
$m^*(\oLS)\leq \mul^*(\ee)$ follows.
The claim is proved.

Next we claim that $\ee$ is an equioscillation point.
Assume for a contradiction that
for some $j\in \{1,\dots,n\}$
we have $m_j^*(\ee)>\mul^*(\ee)$.
Then there is $k_0\in \NN$ such that
$m_j^*(\eta_{k_0},\ee) >\mul^*(\ee)$.
Since $m_j^*(\eta_{k_0},\cdot)$ is continuous
(the kernel functions are singular; see Proposition \ref{prop:mjstarcont} \ref{mjstarcont:parta}),
there is $\delta>0$ such
that for every $\yy\in \oLS$
with $\mathrm{dist}_{\bT^n}(\yy,\ee)<\delta$
one has
$m_j^*(\eta_{k_0},\yy)
>\mul^*(\ee)$, too.

There is $n_0\in \NN$ such that for every $k\geq n_0$
we have
$\mathrm{dist}_{\bT^n}(\ee(\eta_{k}),\ee)<\delta$.
So for $k\geq\max\{k_0,n_0\}$
we can write
\[
m_j^*(\eta_{k},\ee(\eta_k))
\geq
m_j^*(\eta_{k_0},\ee(\eta_k))
>
\mul^*(\ee)=m^*(\oLS)
\geq
\mul^*(\ee(\eta_k))
\geq
\mul^*(\eta_k,\ee(\eta_k)).
\]
This is a contradiction,
since $m_i^*(\eta_{k},\ee(\eta_k))
=\mul^*(\eta_k,\ee(\eta_k))$
for each $i\in \{1,\dots,n\}$.

Finally, if $\ee$ is an equioscillating node system and $K$ is singular,
then necessarily $\ee \in \LS$, as before.

So we proved that there is
a node system $\zz^*:=\ee$,
$\zz^*\in \LS$
for which $\mul^*(\zz^*)=m^*(\oLS)$
and $\zz^*$ is an equiocillating
node system.
\end{proof}

\subsection{A counterexample for nonsingular kernels}

\begin{example}
If $K$ is not a singular kernel,
then there can be
no equioscillation,
and
no maximin node systems.
Moreover, all node systems can be solutions of
the minimax problem.
\end{example}

Let $K\equiv 0$,
and $J(1/\ell):=1-1/\ell$
($\ell\in\NN$),
and $J(t):=0$ otherwise.
Let $n:=2$.
Then there is no equioscillation.
We consider the following  cases.
If $y_1=y_2=0$, then
$m_1^*(\yy)=0$ and $m_2^*(\yy)=1$.
If $y_1=0$, $0<y_2<1$,
then $m_1^*(\yy)=1$ and
$m_2^*(\yy)=1-1/\left\lfloor 1/y_2\right\rfloor$.
If $y_1>0$ and $y_2< 1$, then
$m_1^*(\yy)=1-1/\left\lfloor 1/y_1\right\rfloor$
and $m_2^*(\yy)=1$.
It is easy to see that there is no equioscillating node system.
To verify the assertions about minimax node systems,
note that $\mol^*(\yy)=\max(m_0^*(\yy),m_1^*(\yy))\equiv 1$
hence every node system is a node system
with the minimax value.
By considering node systems $\yy=(0,1/n)$,
we  get that
$\mul^*(\oLS)=1$,
but there is no node system
with $\mul^*(\yy)=1$.

\section{A partial homeomorphism result}

In our earlier papers \cite{TLMS2018, Homeo, Sbornik} on the subject, an outstanding role was played by a new finding, not seen in earlier works of Bojanov \cite{Bojanov} and Fenton \cite{Fenton}. We established, that in case of singular and strictly concave kernels a certain homemorphism exists between admissible node systems and differences of the interval or arc maxima. Since the differences all being zero is equivalent to equioscillation, such a result immediately gives the existence and uniqueness of an equioscillating node system. This entails that proving e.g. that the minimax point is an equioscillating system can be strengthened to say that this equioscillation property \emph{characterizes} the minimax node system. Similarly, if we further prove that a maximin node system is necessarily equioscillating, then it follows that the minimax equals to the maximin, and is attained at that unique point of equioscillation. Furthermore, from the homeomorphism result further consequences could be proved, most importantly about the intertwining of $m_j$, see \cite{Sbornik}.

Above we have proved that minimax and maximin node systems exist and that they are necessarily equioscillating node systems. Therefore it is most natural to try to complete the picture by a corresponding homeomorphism theorem. In fact, even for the torus setup such a homeomorphism theorem was already proved in \cite{TLMS2018}, when there were no weigthts allowed, and where we assumed a few technical assumptions, too. In our present notation Corollary 9.3 of \cite{TLMS2018} runs as follows.

\begin{thm}\label{thm:TLMS-homeo}
Suppose that for each $j = 0,\ldots, n$ the kernel $K_j$ belongs
to $C^2(0,1)$ with $K_j'' < 0$
and satisfies \eqref{cond:infty}.
Let $S := \{ \yy \in \TT^n~:~ 0<y_1<\cdots<y_n<1$ be the open simplex,
while $y_0$ is understood as fixed at $y_0=0$.

Then the difference mapping
$\Phi(\yy):=(m_1^*(\yy)-m_0^*(\yy),\ldots, m_n^*(\yy)-m_{n-1}^*(\yy))$
is a homeomorphism of $S$ onto $\RR^n$.
\end{thm}

Observe that here the $n+1$ \enquote{to be translated} kernels admit a symmetry:
if we rotate a node system $\yy:=(y_0,y_1,\ldots,y_n)\in\TT^{n+1}$
by say $t\in\TT$, then for the new system $\yy_t:=(y_0+t,y_1+t,\ldots,y_n+t)\in\TT^{n+1}$
we get exactly the same vector of arc maxima ${\bf m^*}(\yy_t):=(m^*_0(\yy_t),m^*_1(\yy_t), \ldots, m^*_n(\yy_t)) \in \RR^{n+1}$, hence in particular also the differences of these maxima remain the same as before.
Therefore it was natural to select one copy of these identical systems
by fixing the value of $y_0$--also
this was the only way the repetitions could be discarded and a homeomorphism could hold.

In our current settings, however, there is an outer field $J$, too.
Once the weight, i.e., the field is not constant,
we no longer have this rotational symmetry.
The situation can be compared to the previous case regarding $K_0(\cdot-y_0)$ not as a kernel with fixed $y_0$, but just as an outer field. Theorem \ref{thm:TLMS-homeo} just says that we have a homeomorphism if the field is strictly concave and singular, and all the kernels satisfy the extra assumptions on differentiability etc.
However, in this interpretation one thing constitutes a major difference:
if we take only the $n$ element node systems $(y_1,\ldots,y_n)\in \TT^n$ accompanied by a field (like e.g. $J:=K_0(\cdot-y_0)$), then there will be only $n$ arcs,
determined by the nodes, so that the arc $I^*_n((y_1,\ldots,y_n))=[y_n,y_1]$ will be the union of the former two arcs $I^*_n(\yy)=[y_n,y_0]$ and $I^*_0(\yy)=[y_0,y_1]$.
Similarly, the maximums will form an $n$-dimensional vector with $m_n^*(y_1,\ldots,y_n)$
becoming the maximum of the former two maxima
$ m_n^*(\yy)$ and $m^*_0(\yy)$.
The maximum differences then form an $n-1$-dimensional manifold, and
we can no longer hope for a homeomorphism from the domain of our $n$-dimensional node systems
to this manifold of differences.

To cure this, we may consider $K_0(\cdot-y_0)$ \emph{both} a fixed kernel and also a field
(say writing $\widetilde{K_0}:=\frac12 K_0$ and $J:=\frac12 K_0(\cdot-y_0)$).
Then the result of Theorem \ref{thm:TLMS-homeo} will refer to the original $n+1$ arc maximums and
their $n$ differences, with a valid homeomorphism result.

With this in mind, we prove that
an analogous \enquote{partial} homemorphism theorem remains in effect even
if there is an arbitrary weight, i.e., field.
In addition, we will surpass all the other technical conditions of Theorem \ref{thm:TLMS-homeo}
by the more advanced technology we have developed in \cite{Homeo},
capable of handling even non-differentiable kernels.
Instead of repeating the technical steps of that proof in the torus context,
we directly reduce the statement to results of \cite{Homeo}.
In fact there we have made a substantial effort to formulate and prove results
which can potentially be used even in the periodic case--and that investment brings a profit here enabling us to refer back to them.
Actually, we will use the following, proved for the case of the interval setup as Theorem 18 in \cite{Homeo}.

\begin{theorem}\label{thm:diffperiodic}
Let $K_1,\dots, K_n$ be strictly concave, singular kernel functions fulfilling condition (PM$_0$)
and
let $J$ be a field function satisfying either
$J(0)=\lim_{t\downto 0} J(t)=-\infty$
or $J(1)=\lim_{t\upto 1} J(t)=-\infty$
(or both).

Then the difference function
$\diff(\xx):=
\big(m_1(\xx)-m_0(\xx),\ldots,m_n(\xx)-m_{n-1}(\xx)\big)$ is a
locally bi-Lipschitz homeomorphism between
$Y:=\{(x_1,\ldots,x_n)\in \oSi:
m_k(\xx)>-\infty ~ (k=0,1,\ldots,n) \}\subset [0,1]^n$ and $\bR^n$.
\end{theorem}

Note that here the ordering of nodes is fixed according to the simplex $\oSi$;
and given the singularity condition, all degenerate node systems contain a degenerate arc with $m_k(\xx)=-\infty$,
so in fact the admissible set $Y\subset \oSi$, too.
Also note that a non-admissible node system from $\oSi\setminus Y$ can never be an equioscillating node system, for $\mol(\xx)>-\infty$ excludes equioscillation at the $-\infty$ level.
Thus in particular this entails existence and unicity of an equioscillating node system.

\begin{thm}\label{thm:partial-homeo} Let $n\in \NN$, let $K_0,K_1,\ldots,K_n$ be $n+1$ strictly concave $1$-periodic kernel functions and let $J$ be a $1$-periodic, otherwise arbitrary $n+1$-field function.

For any value $a\in \TT$ denote $Y^*:=Y^*(a):= \{ \yy=(y_0,y_1,\ldots,y_n) \in \oLS \subset \TT^{n+1}~:~ y_0=a, ~ 
m_k(\yy)>-\infty ~(k=0,1,\ldots,n)\}$.

Then the difference function
$\diff^*(\yy):=
\big(m^*_1(\yy)-m^*_0(\yy),\ldots,m^*_n(\yy)-m^*_{n-1}(\yy)\big)$
is a
locally bi-Lipschitz
homeomorphism between $Y^*(a)$ and $\RR^n$.

In particular, for each fixed value $y_0=a$ there exists one unique equioscillating node system of the form $\yy(a)=(a,y_1(a),\ldots y_n(a)) \in \oLS$.
\end{thm}

\begin{proof} First we introduce a new field function $J^*:=J+K_0(\cdot - a)$.
Obviously $J^*$ will then be a $1$-periodic upper bounded function which is nonsingular at more than $n$ points mod 1 (while one finite value, if $J(a)$ was finite,
could be \enquote{killed}
by $K_0(\cdot -a)$). So, $J^*$ is an $n$-field function, and in view of the singularity condition on $K_0$ it also satisfies the extra singularity equation
$\lim_{t \to a} J^*(t)=-\infty$.

Interpreting the $K_j$ and $J^*$ as defined on the torus, we now transfer them to the interval $[0,1]$ via the covering mapping $\pi_a$. We put $\widetilde{J}(r):=J^*(\pi_a(r))$, which then becomes an $n$-field function on $[0,1]$ satisfying $\lim_{r\downto 0} \widetilde{J} = \lim_{r \upto 1} \widetilde{J} =-\infty$. Also we put $\widetilde{K_j}(r):=K_j(\pi_a(r)-a)$ for all $j=1,\ldots,n$. Obviously these are singular kernel functions, Which are strictly concave, too.

We apply the above Theorem \ref{thm:diffperiodic} to this new system. Let us see what are the arising node systems $\xx \in Y$. First, $\xx \in \oS$ transfers to the cyclic ordering $a \clesse \pi_a(x_1) \clesse \dots \clesse \pi_a(x_n)$, so that writing $\yy:=(y_0,y_1,\ldots,y_n)$ with $y_0:=a$ and $y_j:=\pi_a(x_j)$ ($j=1,\ldots,n$), the ordering condition becomes $\yy \in \oLS$.

Second, the non-singularity conditions for $\xx \in Y$ translate to those for $\yy \in Y^*(a)$,
because $\pi_a : I_k(\xx) \leftrightarrow I_k^*(\yy)$ and
$$
\widetilde{F}(\xx,r)=\widetilde {J}(r)+\sum_{k=1}^n \widetilde{K_k}(r-x_k)= J(t)+K_0(t-a) + \sum_{k=1}^n K_j(t-y_k) \qquad (t:=\pi_a(r))
$$
makes a one-to-one correspondence between sum of translates function values on corresponding points of any $I_k(\xx)$ resp. $I_k^*(\yy)$; in particular,
we find $m_k(\xx)=m_k^*(\yy)$, ($k=0,\ldots,n$), and an arc $I^*_k(\yy)$ will be singular if and
only if the corresponding interval $I_k(\xx)$ was.

So we have seen that when $\yy$ runs $Y^*(a)$ then $\xx$ runs $Y$, and the correspondence is one-to-one.
Moreover, degenerate points do not satisfy the nonsingularity condition, hence
$Y \subset \Si$ and $Y^*(a) \subset \LS$.
This means that between node systems $\xx$ and $\yy$,
the mapping $\pi_a$ and even its inverse $\pi_a^{-1}$ acts continuously,
given that $\pi_a^{-1}$ is applied only to $y_1,\ldots,y_n$ off $a$.

Summing up, $\Phi^*(\yy)=\Phi(\pi_a^{-1}(\yy))$, and even this composition mapping is continuous, moreover, it maps one-to-one to $\RR^n$. Obviously its inverse $\Phi^{-1}\circ \pi_a$ is continuous, too, once $\Phi^{-1}$ was, so the mapping is a homeomorphism between $Y^*(a)$ and $\RR^n$. We also get the bi-Lipschitz property from that of $\Phi$.

Finally, uniqueness of an equioscillating node system with given fixed $y_0=a$ follows from the homeomorphism result for $Y^*(a)$
(where we get exactly one system with all differences zero), and from the
fact that outside $Y^*(a)$ all node systems provide some singular value $m^*_k(\yy)=-\infty$,
while equioscillation cannot take place on that level,
for $\mol^*(\yy)  \in \RR$, always.
\end{proof}

Having proved the above, we may try to progress towards a Bojanov-type characteristaion result. We already know that for the global minimax point there is equioscillation; and we now established that for each fixed value of $y_0$ there is exactly one equioscillating node system. So, writing $\ff(a)= (a, y_1(a),\ldots,y_n(a))$ for this unique equioscillation point with $y_0=a$, it suffices to look for the global minimum of $\mu(a):=\mol^*(\ff(a))$ over $a \in \TT$. The question is if these equioscillation values are always the same -- as we have seen in \cite{TLMS2018} when $J\equiv 0$ -- or if $\mu(a)$ is non-constant. Unicity of equioscillating \emph{node systems} cannot hold (there is one for each fixed value $a$ of the first coordinate $y_0$), but one may hope for unicity of the \emph{equioscillation value}.

In Section \ref{sec:counterexample}, however, we will see that
even this modified hope is deluded.

\section{Counterexamples -- equioscillation does not characterize minimax or maximin, majorization occurs, minimax can be smaller than maximin}\label{sec:counterexample}

Bojanov's Theorem \cite{Bojanov} in the classical algebraic polynomial setting included an important characterization statement, too: the extremal minimax system was \emph{characterized} by the equioscillation property.
That is, there was exactly one equioscillating node system, which was necessarily the minimax point.

Fenton's classical theorem added another statement to the theory (in his context):
he also proved that the unique maximin point equals to the (unique) minimax point (and hence the maximin and minimax values are equal, too).

In the weighted algebraic setting we proved similarly strong results in \cite{Sbornik}. Also, we found the same for the unweighted trigonometric setting in \cite{TLMS2018}, save the trivial free rotation (which was discarded by indexing from $0$ and fixing the value of $y_0$ in $\TT$).

In this section we explore examples which show that we cannot expect as many results as for the interval case or for the torus without weights.
Basically, we will show that $m(\LS)>M(\LS)$ does occur.

Our examples will use a singular kernel,
so  by our above results minimax and maximin points exist,
moreover, they are equioscillation points.
Therefore, the examples also mean that there are different equioscillation values, hence strict majorization occurs even between equioscillating node systems.
Furthermore, we will choose the kernel to be the standard log-sine kernel $K(t):=\log|\sin(\pi t)|.$
This means that even for the classical trigonometric polynomial case one cannot expect any better results for general weights.

\subsection{A counterexample with majorization}

\begin{example}
\label{ex:counterexample}
Set $K(t):=\log|\sin(\pi t)|$, $n=2$, $\nu_1=\nu_2=1$, and
$J(t)=0$ on $\{0\} \cup [1/2,1)$ and $J(t)=-\infty$ on $(0,1/2)$.

Then,
the minimax and maximin values on the
\enquote{cyclic simplex}
$\LS$ are $m^*(\oLS)=-2\log(2)$,
$M^*(\oLS)=-\log(2)$, respectively.

Moreover, for any $\lambda\in [-2\log(2),-\log(2)]$
there is an equioscillating node system
$\yy\in \LS$
with $\lambda=m_1^*(\yy)=\ldots=m_n^*(\yy)$.
\end{example}

In the following we will determine all equioscillation values.

We are to minimize $\mol^*(\yy)$ (and maximize $\mul^*(\yy)$) on $\yy \in \TT^2$.
By relabeling, if necessary,
we can assume $\yy=(y_1,y_2)$ with $0\le y_1\le y_2<1$.

First, we make the following simple observations
about the behavior of the pure sum of translates function.
From the evenness of $K$ it follows that $f(\yy,\cdot)$ behaves
symmetrically on the two intervals before and after the midpoint $\frac{y_1+y_2}{2}$
on $I_1^*(\yy)\sim [y_1,y_2]$, and similarly for the other arc $I_2^*(\yy)$.
In particular, $f(\yy,\cdot)$ is strictly monotone increasing on $(y_1,\frac{y_1+y_2}{2})$ and decreasing on $(\frac{y_1+y_2}{2}, y_2)$.
Similarly, it is strictly increasing on $(y_2,\frac{y_1+y_2}{2}+\frac{1}{2})$
and strictly decreasing on
$(\frac{y_1+y_2}{2}+\frac{1}{2},1+y_1)$.
Moreover, $f$ is maximal on $I_i(\yy)$ at the midpoint, and
if the length of this interval is $\ell:=|I_i(\yy)|$,
the length of the arc $I_i(\yy)$,
then its value is $2K(\ell/2)$.
As for $F(\yy,\cdot)$, it follows that on any of the arcs $I_i^*(\yy)$
it attains its maximum on the point(s) of the arc which have $J(t)=0$ and are closest to the midpoint among those.
Also note that among the two midpoints, which are exactly of distance $1/2$, only one can belong to the singular set $X_J=(0,1/2)$.

\medskip
In the following we use the variables
$x=y_1+y_2$, $z=y_2-y_1$,
by which we can express $y_1=\frac{x-z}{2}$, $y_2=\frac{x+z}{2}$.

Note that we cannot have $0\le y_1\le y_2\le1/2$,
for then $m_1^*(\yy)=-\infty$, which cannot be an equioscillation value, given that $\mol^*(\yy)>-\infty$.
Thus $y_2\ge 1/2$, and we have $x\ge y_2\ge 1/2$, too.

So, let the first case be $1/2\le x< 1$.
Then, the midpoint $x/2$ of the arc $I_1^*(\yy)$ lies in $[1/4,1/2)$,
which is in the singular set $(0,1/2)$,
hence the maximum will be attained at the closest nonsingular point of the arc,
which is $1/2$.
That is,
$m_1^*(\yy)=F(\yy,1/2)=f(\yy,1/2)$.
Further,
$m_2^*(\yy)=F(\yy,x/2+1/2)=f(\yy,x/2+1/2)$.

We will use the identity
\begin{align} \notag
f(\yy,t)&=
f\left(\left(\frac{x-z}{2},
\frac{x+z}{2}\right), t\right)
=
\log\left|\sin \pi
\left(t-\frac{x-z}{2}\right)\,
\sin \pi
\left(t-\frac{x+z}{2}\right)
\right|
\\& =
-\log(2)+\log\left|
\cos(\pi z) -\cos \pi(2t -x)
\right|\qquad (t \in \RR).
\label{eq:identity}
\end{align}

With this,
the equation $f(\yy,1/2)=f(\yy, x/2+1/2)$
can be rewritten as
\begin{equation*}
\left|\cos(\pi z) -\cos\pi \left(2\frac{1}{2}-x\right)\right|
=
\left|\cos(\pi z) -\cos\pi \left(2\left(\frac{x}{2}+\frac{1}{2}\right)-x\right)\right|.
\end{equation*}
The right hand side is $1+\cos(\pi z)$.
For the sign of the left hand side we observe $x+z=2y_2\ge 1$,
hence $1-x\le z \le x \le 1$,
so by monotonicity of $\cos(\pi s)$ for $0\le s \le 1$,
we get
$\left|\cos(\pi z) -\cos\pi \left(2\frac{1}{2}-x\right)\right|=\cos(\pi(1-x))-\cos(\pi z)$.

So we are led to the equation
\begin{equation*}
\cos \pi(1-x)\ -\cos(\pi z)
=
\cos(\pi z) +1,
\end{equation*}
and solving it for $\cos(\pi z)$
yields
\begin{equation*}
\cos(\pi z)=
\frac{\cos\pi(1-x) -1}{2}.
\end{equation*}
The condition $1-x\le z$ is satisfied,
since $\cos\pi(1-x)\ge (\cos\pi(1-x)-1)/2$ in general.
By simple steps, the condition $z\le x$ is
equivalent to $\cos(\pi x)\le -1/3$.
So if $1/2\le x\le \beta_0:=\arccos(-1/3)/\pi\approx 0.608$,
then $z$ does not satisfy
$z\le x$.

The arising equioscillation value is
\begin{multline*}
m_2^*\left(\left(\frac{x-z}{2},\frac{x+z}{2}\right)\right)
= f(\yy,x/2+1/2)
\\
=-\log(2)+\log\left|
\frac{\cos\pi(1-x)\;-1}{2} +1
\right|
\\ =
-2\log(2)+\log\left(1-\cos(\pi x)\right)
\end{multline*}
for $\beta_0\le x\le 1$.

The next case is $1\le x< 3/2$.
Since $x+z=y_1+y_2\le 2$, we have $0\le z\le 2-x$, too.
Now $x/2 \notin X_J$,
while $x/2+1/2 \in X_J$
(except for $x=1$ and $x/2+1/2=1$),
with the closest nonsingular point from $I_2^*(\yy)$ being $1$ (remaining valid even in case $x=1$).
Hence, $m_1^*(\yy)=f(\yy,x/2)$
and $m_2^*(\yy)=f(\yy,1)$.
Now, again by \eqref{eq:identity},  $f(\yy,x/2)=f(\yy, 1)$
is equivalent to
\begin{equation}\label{simp:fourthcase}
\left|\cos(\pi z) -\cos\pi \left(2\frac{x}{2}-x\right)\right|
=
\left|\cos(\pi z) -\cos\pi \left(2\cdot 1 -x\right)\right|.
\end{equation}
Here the left hand side is
$|\cos(\pi z)-1|=1-\cos(\pi z)$.
Also, $0\le z\le 2-x \le 1$, so
$\cos(\pi z)\ge \cos \pi(2-x)$,
and \eqref{simp:fourthcase}
can be written as
\begin{equation*}
1-\cos(\pi z)
=\cos(\pi z) - \cos\pi(2- x).
\end{equation*}
Solving it for $\cos(\pi z)$
we are led to
\begin{equation*}
\cos(\pi z)=
\frac{1+\cos\pi x}{2}.
\end{equation*}
The condition $0\le z\le 2-x \le 1$
on $z$ is equivalent to
$1\ge \cos(\pi z)\ge\cos\pi(2-x)=\cos(\pi x)$
and hence is obviously satisfied.

The equioscillation value,
again depending only on
$x$, is found again to be
\begin{multline*}
m_2^*\left(\left(\frac{x-z}{2},\frac{x+z}{2}\right)\right)
= f(\yy,x/2)
\\
=-\log(2)+\log\left|1-
\frac{1+\cos(\pi x)}{2}
\right|
\\ =
-2\log(2)+
\log\left(1-\cos(\pi x)\right).
\end{multline*}

Finally, let $3/2\le x<2$.
This means that $y_1\ge 1/2$, for $x=y_1+y_2\le y_1+1$.
Then again, $m_1^*(\yy)=f(\yy,x/2)$, for $x/2 \in [1/2,1]$.
However, in $I^*_2(\yy)$ the closest non-singular point to the midpoint $x/2+1/2\equiv x/2-1/2 \bmod 1$ will be $1/2$, and we will get
$m_2^*(\yy)=f(\yy,1/2)$.
Therefore, the equioscillation equation becomes
$f(\yy,x/2)=f(\yy,1/2)$
and, by \eqref{eq:identity}, it is
\begin{equation*}
|\cos(\pi z)-\cos \pi\left(2\frac{x}{2}-x\right)|
=
|\cos(\pi z)-\cos \pi\left(2\frac{1}{2}-x\right)|.
\end{equation*}
The left hand is $1-\cos(\pi z)$
and the right hand side is
$|\cos(\pi z)-\cos \pi x|$.
Since $3/2\le x <2$ and $z+x=2y_2\le 2$ implies $0\le z \le 2-x<1/2$,
here both terms
are nonnegative, and the equation becomes
\begin{equation*}
1-\cos(\pi z)
=
\cos(\pi z) + \cos \pi x .
\end{equation*}
Solving it for $\cos(\pi z)$
we get
\begin{equation*}
\cos(\pi z)=
\frac{1-\cos(\pi x)}{2}.
\end{equation*}
Again, we verify that $0\le z\le 2-x$,
which is equivalent to
$1\ge \cos(\pi z)\ge \cos\pi(2-x)$.
The second inequality is equivalent to
$\cos(\pi x)\le 1/3$,
and it holds for $x\in[3/2,2]$
if and only if $x\le 2-\arccos(1/3)/\pi =1+\beta_0$,
$1+\beta_0\approx 1.608$.
So $z\le 2-x$ if and only if
$x\in[3/2,1+\beta_0]$.

The equioscillation value  is
\begin{multline*}
m_2^*\left(\left(\frac{x-z}{2},\frac{x+z}{2}\right)\right)
= f(\yy,1)
\\
=-\log(2)+\log\left|1-
\frac{1-\cos(\pi x)}{2}
\right|
\\ =
-2\log(2)+
\log\left(1+\cos(\pi x)\right).
\end{multline*}

Now we can collect the obtained equioscillation values
coming from all the three cases:
\begin{equation*}
m_1^*(\yy)=m_2^*(\yy)=-2 \log 2 +
\begin{cases}
\log(1-\cos (\pi x))
\qquad & (\beta_0\le x<3/2)
\\
\log(1+\cos (\pi x))
\qquad & (3/2\le x<1+\beta_0) \end{cases}
\end{equation*}
It is minimal when
$x=3/2$ (in this case $z=1/3$)
and it is maximal when
$x=1$ (in this case $z=1/2$).
The corresponding values are
$-2\log(2)\approx -1.386$
and $-\log(2)\approx -0.693$.

Hence we get
$M(\oLS)= -2\log(2)$
and
$m(\oLS)= -\log(2)$.

\subsection{A modified counterexample with continuous field function}

In this subsection we sketch a counterexample
which is a modification of the previous one, but with a continuous field function.

First, we set a new external
field function $\newJ$, and then
we compute the arc maxima for the two extremal
node systems from the previous counterexample.
We will find that the arc maxima $m_i^*(\yy)$ of those two node systems will not change when we replace $\widetilde{J}$ for $J$, whence they will still be equiocillating
with the same equioscillation values $-\log 2$ and $-2\log 2$, respectively.

Let $\alpha > 4\pi$ be fixed
and let $\newJ$ be $0$ on $[1/2,1]$,
$-\alpha t$ on $[0,1/4)$
and $\alpha (t-1/2)$ on $[1/4,1/2)$; further, let $\newJ$ be extended $1$-periodically to $\RR$.

Regarding the $x=1$, $z=1/2$ maximin configuration,
we have $y_1=1/4$, $y_2=3/4$.
It is easy to check that $f(\yy,t)$ is
strictly monotone decreasing on $(0,1/4]$
and on $[1/2,3/4)$
and it is strictly monotone increasing
on $(1/4,1/2]$
and on $(3/4,1]$.
Hence, $\newF(\yy,t)=\newJ(t)+f(\yy,t)$
is the same monotone in these intervals.
Therefore $\newmstar_1(\yy)=\newF(\yy,1/2)=F(\yy,1/2)=f(\yy,1/2)=-\log(2)$
and $\newmstar_2(\yy)=\newF(\yy,1)=F(\yy,1)=f(\yy,1)=-\log(2)$.

Regarding the $x=3/2$, $z=1/3$ minimax configuration,
we have $y_1=7/12$, $y_2=11/12$.
It is easy to check that $f(\yy,t)$ is
strictly monotone increasing (and concave)
on $(-1/12,1/2]$
and $f'(\yy,0)=4\pi$.
Adding $\newJ$ to it, we see that
$\newF(\yy,\cdot)$ is strictly monotone increasing on $(-1/12,0]$
and strictly monotone decreasing on $[0,1/4)$.
Using the symmetry of $\newF(\yy,\cdot)$
with respect to $1/4$,
$\newF(\yy,\cdot)$ is strictly monotone increasing on $(1/4,1/2]$
and strictly monotone decreasing on $[1/2,7/12)$.
Moreover, $\newmstar_1(\yy)=\newF(\yy,0)=\newF(\yy,1/2)=f(\yy,0)=-2\log(2)$.
Computing $\newmstar_2(\yy)$ is simpler: for
$t\in I_2^*(\yy)=(7/12,11/12) \subset [1/2,1]$, $\newF(\yy,t)=f(\yy,t)=F(\yy,t)$, and $\sup f(\yy,\cdot)$ remains the same as before.
Therefore, $\newmstar_2(\yy)=\newF(\yy,3/4)=f(\yy,3/4)=-2\log(2)$.

Summing up, we obtained that
$\widetilde{m}(\oLS)\le -2\log(2)$
and $\widetilde{M}(\oLS)\ge -\log(2)$.
Also, taking into account that
$\widetilde{J}\ge J$, hence $\newF\ge F$,
we also know that
$\widetilde{m}(\oLS)\ge m(\oLS)$
and $\widetilde{M}(\oLS)\ge M(\oLS)$.
Therefore,
$\widetilde{M}(\oLS)= -2\log(2)$, too.

Even if we don't proceed to compute the exact maximin value, too, it is clear from what has already been done that in this example $\widetilde{m}(\oLS)\ge -\log(2) > \widetilde{M}(\oLS)=-2 \log 2$, hence the same phenomenon takes place for the continuous, finite kernel $\newJ$ as before for the kernel $J$.

\section*{Acknowledgment}

This research was supported by project
TKP2021-NVA-09.
Project no.~TKP2021-NVA-09
has been implemented with the support provided by the Ministry of Innovation
and Technology of Hungary from the National Research, Development and
Innovation Fund, financed under the TKP2021-NVA funding scheme.

The work of Sz.~Gy.~Révész was supported in part by Hungarian National Research, Development and Innovation Fund project \# K-119528.

\medskip

\noindent
\hspace*{5mm}
\begin{minipage}{\textwidth}
\noindent
\hspace*{-5mm}Béla Nagy\\
 Department of Analysis,\\
 Bolyai Institute, University of Szeged\\
 Aradi vértanuk tere 1\\
  6720 Szeged, Hungary\\
\end{minipage}

\medskip

\noindent
\hspace*{5mm}
\begin{minipage}{\textwidth}
\noindent
\hspace*{-5mm}
Szilárd Gy.{} Révész\\
 Alfréd Rényi Institute of Mathematics\\
 Reáltanoda utca 13-15\\
 1053 Budapest, Hungary \\
\end{minipage}

\end{document}